\documentclass[a4paper, intlimits, 10pt,reqno]{amsart}

\usepackage[T1]{fontenc}
\usepackage{float}

\usepackage{times}
\usepackage[english]{babel}
\usepackage[T1]{fontenc}
\usepackage{enumitem}

\usepackage{graphicx}

\usepackage{subcaption}
\usepackage{tabularx}
\usepackage{color}
\usepackage{eurosym}
\usepackage[numbers, sort&compress]{natbib}
\usepackage{bm}

\DeclareMathOperator*{\esssup}{ess\,sup}

\usepackage[colorlinks=true,urlcolor=blue, citecolor=blue,linkcolor=blue,
linktocpage,pdfpagelabels, bookmarksnumbered,bookmarksopen]{hyperref}

\usepackage{esint}

\usepackage{amsthm, amssymb, amsfonts}
\usepackage{amsmath}

\usepackage{accents}

\usepackage{multirow}

\newtheorem{theorem}{Theorem}[section]
\newtheorem{corollary}[theorem]{Corollary}

\newtheorem{lemma}[theorem]{Lemma}

\newtheorem{definition}[theorem]{Definition}

\newtheorem{remark}[theorem]{Remark}

\numberwithin{equation}{section}

\begin{document}

\title{Extremal of Log\--Sobolev Functionals and Li-Yau Estimate on $\mathrm{RCD}^*(K,N)$ Spaces}

\date{\today}

\thanks{National Science Foundation of China, Grants Numbers: 11971310 and 11671257 are gratefully acknowledged.}

\author{Samuel Drapeau}
\address{Shanghai Jiao Tong University, School of Mathematical Sciences \& Shanghai Advanced Institute for Finance, Shanghai, China}
\email{sdrapeau@saif.sjtu.edu.cn}
\urladdr{http://www.samuel-drapeau.info}

\author{Liming Yin}
\address{Shanghai Jiao Tong University, School of Mathematical Sciences, Shanghai, China}
\email{gacktkaga@sjtu.edu.cn}

\begin{abstract}
    In this work, we study the extremal functions of the log-Sobolev functional on compact metric measure spaces satisfying the $\mathrm{RCD}^*(K,N)$ condition for $K$ in $\mathbb{R}$ and $N$ in $(2,\infty)$.
    We show the existence, regularity and positivity of non-negative extremal functions.
    Based on these results, we prove a Li-Yau type estimate for the logarithmic transform of any non-negative extremal functions of the log-Sobolev functional.
    As applications, we show a Harnack type inequality as well as lower and upper bounds for the non-negative extremal functions.
\end{abstract}

\maketitle
\section{Introduction}
In this work, we study the extremal functions of the following variational problem of the log-Sobolev functional 
\begin{equation}\label{eq:intro_log_sobolev_functional}
    \lambda(\alpha_1,\alpha_2):=
    \inf_{f\in\mathcal{F}}\int_{X}\left(|\nabla f|^2 +\alpha_1 f^2 - \frac{\alpha_2}{2} f^2 \log f^2\right)dm 
\end{equation}
where $\alpha_1$ and $\alpha_2$ are in $\mathbb{R}$ and $\mathcal{F}:=\{f\in W^{1,2}: \|f\|_{L^2}=1\}$.
Of interest are the existence, regularity, positivity of non-negative extremal functions, and analytic results such as Li-Yau or Harnack type estimates.
While those results are well-known in the smooth Riemannian setting, it seems natural to ask whether they can be extended and in which form to more general non-smooth metric spaces.

While the log-Sobolev inequality has vast applications in different branches of math\-e\-mat\-ics---see \citet*{gross1975logarithmic,otto2000generalization, bakry2014analysis}---studying its extremal functions is important on its own right. 
For example, \citet*{zhang2012extremal} shows that in the noncompact smooth manifold, the geometry of the manifold at infinity will affect the existence of extremal of the log-Sobolev functional and Perelman's $W$\--entropy.
Using these points, the author further shows that under the mild assumptions, noncompact shrinking breathers of Ricci flow are gradient shrinking solitons.
Very recently, the extremal functions of the log-Sobolev inequality are used together with the needle decomposition technique by \citet*{ohta2020equality} to show some rigidity result of the underlying weighted Riemannian manifold. 

From the viewpoint of the underlying space, starting with the works of \citet*{sturm2006geometry,sturm2006geometry2} and \citet*{lott2009ricci}, the synthetic notion of Ricci curvature---referred to as $\mathrm{CD}(K,N)$ condition---bounded from below by $K$ in $\mathbb{R}$ and the dimension bounded from above by $1\leq N\leq \infty$ on a general metric measure space without having a smooth structure was introduced and developed greatly in the last decade.
The key property of this notion is that it is compatible with the smooth Riemannian setting and stable with respect to the measured Gromov\--Hausdorff convergence so that it includes Ricci limit spaces and Alexandrov spaces.
Later, to rule out the Finsler geometry, the finer $\mathrm{RCD}(K,\infty)$ condition was introduced by \citet*{ambrosio2014metric} and the finite dimensional counterpart $\mathrm{RCD}^*(K,N)$ condition was introduced and studied in \cite{erbar2015equivalence,ambrosio2019nonlinear,ambrosio2016bakry}.
Recently, building upon the abstract module theory, the first and second order differential structure on $\mathrm{RCD}(K,\infty)$ spaces was developed by \citet*{gigli2018nonsmooth}, and finer geometric results such as the rectifiability of $\mathrm{RCD}(K,N)$ spaces were studied by \citet*{mondino2019structure}. 

With these analytic tools, the geometric analysis on metric measure spaces satisfying the synthetic Ricci curvature condition also developed quickly.
For instance, Li\--Yau\--Hamilton type inequalities for the heat flow \cite{garofalo2014li,jiang2015li_yau,jiang2016hamilton} and the localized gradient and a local Li\--Yau estimate for the heat equation \cite{huang2020localized,zhang2016local}.
In particular, \citet*{zhang2016local} develop an Omori\--Yau type maximum principle on the $\mathrm{RCD}^*(K,N)$ space and use it to show a pointwise Li-Yau type estimate for locally weak solutions of the heat equation which may not have the semigroup property.

Motivated by these works, we study the extremal functions of the log-Sobolev functional \eqref{eq:intro_log_sobolev_functional} on more general metric measure spaces.
In particular, we are interested in whether analytic results such as Li\--Yau type estimates for the non-negative extremal functions of the log-Sobolev functional holding on smooth Riemannian manifolds can be extended to non-smooth metric measure spaces, in particular, those satisfying the synthetic Ricci curvature condition.

To do so, one of the key points is to show the existence, boundedness, regularity and positivity of the non-negative extremal functions of the log-Sobolev functional.
Our first main result, Theorem \ref{thm:var_general}, states that the log-Sobolev functional \eqref{eq:intro_log_sobolev_functional} with $\alpha_2>0$ on a compact metric measure space satisfying the $\mathrm{RCD}^*(K,N)$ condition with $K$ in $\mathbb{R}$ and $N$ in $(2,\infty)$ admits non-negative extremal functions which satisfy certain Euler-Lagrange equation.
Moreover, we show that all the non-negative extremal functions are bounded, Lipschitz continuous and bounded away from 0.

We remark that while the existence and Euler-Lagrange equation problems are quite standard and similar to the smooth compact cases solved by \citet*{rothaus1981logarithmic}, several problems arise on metric measure spaces.
For instance, the positivity of the non-negative extremal functions \cite[page 114]{rothaus1981logarithmic} is shown by relying heavily on the underlying smooth differential structure of the Riemannian manifold, the polar coordinates and the exact asymptotic volume ratio near the pole so that the problem can be reduced to a one-dimensional ODE problem.
However, these smooth structures are lost on general metric measure spaces.
While for $\mathrm{RCD}^*(K,N)$ spaces, the polar decomposition still works by using the ``needle decomposition'' generalized to essentially non-branching $\mathrm{CD}(K,N)$ spaces by \citet*{cavalletti2017sharp_isomperimetric} (see also \cite{cavalletti2012local}), the similar asymptotic volume ratio analysis seems to fail without further assumptions on the underlying metric measure space.
To overcome this difficulty, we make use of a maximum principle type argument for the De Giorgi class on some local domain proved by \citet*{kinnunen2001regularity}, to show that non-negative extremal functions are either bounded away from $0$ or vanish on the whole space.
This method works in very general metric measure spaces supporting the local doubling property and weak Poincaré inequality, which is in particular the case for $\mathrm{RCD}^*(K,N)$ spaces.

Our second main result is Theorem \ref{thm:estimate_Li_Yau}.
Based on the regularity and positivity results obtained above, we recover a Li-Yau type estimate for the logarithmic transform of all non-negative extremal functions of \eqref{eq:intro_log_sobolev_functional}.
More precisely, for any non-negative extremal functions $u$, it holds that
\begin{equation}\label{eq:intro_2}
    |\nabla v|^2 + (\alpha_2 -\beta K)v
    \leq
    \frac{N\alpha_2(1-\beta)}{4\beta}\left(1-\frac{\beta((2-\beta)K-\alpha_2)}{2\alpha_2(1-\beta)}\right)^2\quad \text{$m$-a.e.,}
\end{equation}
for any $\beta$ in $(0,1)$ and $v=\log u+(\lambda-\alpha_1)/\alpha_2$.
The same estimate for the smooth Riemannian case was shown by \citet*{wang1999harnack}, where the argumentation relies on the pointwise Bochner formula and a pointwise characterization of local maximum points of the smooth function in the left-hand side of \eqref{eq:intro_2}.
However, in the $\mathrm{RCD}$ setting, neither the function in the left-hand side of \eqref{eq:intro_2} is smooth and pointwise defined, nor the pointwise Bochner formula is available.
To overcome the difficulty, we follow a similar argument as in \cite{zhang2016local} by using an Omori\--Yau type maximum principle.
Note that, to avoid the sign problem of $|\nabla v|^2 +(\alpha_2-\beta K)v$, we use a different auxiliary function $\phi$ from those in \cite[Theorem 1.4]{zhang2016local}, which are constructed from the distance function so that they have the measure-valued Laplacian.
Our construction is based on the ``good'' cut-off functions from \citet*[Lemma 3.1]{mondino2019structure}, which are smoother than those in \cite{zhang2016local} and has the $L^2$-valued Laplacian.
Furthermore, for our purpose, we slightly extend the Omori\--Yau maximum principle in \cite{zhang2016local}, which holds on $\mathrm{RCD}^*(K,N)$ spaces with $K\in \mathbb{R}$ and $N\geq 1$, to proper $\mathrm{RCD}(K,\infty)$ spaces.
While most arguments are similar, our proof follows the so-called ``Sobolev-to-Lip'' property, a property shared by all $\mathrm{RCD}$ spaces, rather than the ``weak maximum principle'' from \cite{zhang2016local}.

Finally, we provide some applications of the regularity and positivity results and the Li\--Yau type estimate.
We show a Harnack type estimate as well as lower and upper bounds of the non-negative extremal functions of \eqref{eq:intro_log_sobolev_functional} depending only on the geometry of the space. 
These generalize the results in \cite{wang1999harnack} proved in the Riemannian setting.
Using the weak Bochner inequality, we also show that all non-negative extremal functions are constant when $0< \alpha_2 \leq KN/(N-1)$, which is well-known in the smooth setting.

The paper is organized as follows: in Section 2, we introduce the notations and definitions about metric measure spaces and $\mathrm{RCD}$ conditions as well as the analytic results needed later.
In Section 3, we study the variational problem and show the existence, regularity and positivity of non-negative extremal functions of \eqref{eq:intro_log_sobolev_functional}.
Section 4 is dedicated to the Li\--Yau type estimate for the non-negative extremal functions.
In Section 5, we present some applications of the previous results.
Finally, in an Appendix, we prove the Omori\--Yau type maximum principle for proper $\mathrm{RCD}(K,\infty)$ spaces.

\section{Preliminary and notations}
We briefly introduce the terminologies and notations related to calculus.
For more details, we refer the readers to \citep{gigli2020lectures, gigli2018nonsmooth}.

Throughout this work, we denote by $(X,d,m)$ a metric measure space where $(X,d)$ is a complete, separable and proper metric space and $m$ is a non-negative Radon measure with full support which is finite on every bounded and measurable set.
We denote by $B(x,r)$ and $\bar{B}(x,r)$ the open and closed metric balls centered at $x$ in $X$ with radius $r>0$, respectively.
By $L^p:=L^p(X,m)$ for $1\leq p\leq \infty$ we denote the standard $L^p$ spaces with $L^p$-norm $\|\cdot\|_{p}$.
By $L^p_{loc}$ we denote those measurable functions $f:X \to \mathbb{R}$ such that $f\chi_B$ is in $L^p$ for any bounded and measurable subset $B$ of $X$, where $\chi_B$ denote the indicator function of the set $B$.
Let further $\mathrm{Lip}:=\mathrm{Lip}(X)$, $\mathrm{Lip}_{loc}:=\mathrm{Lip}_{loc}(X)$, and $\mathrm{Lip}_{bs}:=\mathrm{Lip}_{bs}(X)$ be the spaces of real-valued functions $f:X \to \mathbb{R}$ which are Lipschitz, locally Lipschitz, and Lipschitz with bounded support, respectively.
For $f$ in $\mathrm{Lip}_{loc}$, we denote by $\mathrm{lip}(f)$ the local Lipschitz constant, or slope, defined for any $x$ in $X$ as
\begin{equation*}
    \mathrm{lip}f(x):=\limsup_{y\rightarrow x}\frac{|f(y)-f(x)|}{d(y,x)},
\end{equation*}
if $x$ is not isolated and $\mathrm{lip}f(x)=0$ if $x$ is isolated.

\subsection{Cheeger Energy, Laplacian and Calculus Tools}
The Cheeger energy is a $L^2$\--lower\--semicontinuous and convex functional $\mathrm{Ch}:L^2\rightarrow [0,\infty]$ defined as
\begin{equation*}
    \mathrm{Ch}(f)
    :=
    \inf\left\{\liminf\frac{1}{2}\int_{X}\left(\mathrm{lip}(f_n)\right)^2dm \colon (f_n)\subseteq \mathrm{Lip}\cap L^2, \|f_n-f\|_2\rightarrow 0 \right\}.
\end{equation*}
The domain of $\mathrm{Ch}$ is a linear space denoted by $W^{1,2}:=W^{1.2}(X)$ and is called the Sobolev space.
For $f$ in $W^{1,2}$, we identify the canonical element $|\nabla f|$ called the minimal relaxed gradient as the unique element with minimal $L^2$-norm, also minimal in the $m$-a.e. sense, in the set
\begin{equation*}
    \left\{G\in L^2: G = \lim \mathrm{lip} f_n \text{ in }L^2\text{ for some } (f_n)\subseteq \mathrm{Lip} \text{ such that } f_n\rightarrow f \text{ in $L^2$}\right\},
\end{equation*}
which provides an integral representation $\mathrm{Ch}(f)=\frac{1}{2}\int_{X}|\nabla f|^2 dm$.
The Sobolev space equipped with the norm $\|f\|^2_{W^{1,2}}:=\|f\|^2_{2}+2\mathrm{Ch}(f)$ is a Banach space and is dense in $L^2$, see \cite[Proposition 4.1]{ambrosio2014calculus}.
We further denote by $W^{1,2}_{loc}:=\{f\in L^2_{loc}: \eta f\in W^{1,2} \text{ for any } \eta\in \mathrm{Lip}_{bs}\}$ the space of local Sobolev functions, and define the minimal relaxed gradient as $|\nabla f|:=|\nabla (\eta f)|$ $m$-a.e. on $\{\eta=1\}$ for $f$ in $W^{1,2}_{loc}$ where $\eta$ is in $\mathrm{Lip}_{bs}$.

We say that $(X,d,m)$ is \emph{infinitesimally Hilbertian} if the Cheeger energy is a quadratic form, or equivalently, $W^{1,2}$ is a Hilbert space.
Under these assumptions, it can be proved that for any $f$ and $g$ in $W^{1,2}$, the limit
\begin{equation*}
    \langle \nabla f, \nabla g \rangle
    :=
    \lim_{\varepsilon\rightarrow 0}\frac{|\nabla(f+\varepsilon g)|^2-|\nabla f|^2}{2\varepsilon}
\end{equation*}
exists in $L^1$ and it is a bilinear form from $W^{1,2}\times W^{1,2}$ to $L^1$, see \textcolor{blue}{\citep{ambrosio2014metric}}.
\begin{definition}\label{def:Laplacian}
    Let $(X,d,m)$ be infinitesimally Hilbertian.
    \begin{itemize}[fullwidth]
        \item \textbf{Laplacian:} We say that $f$ in $W^{1,2}$, is in the domain of the Laplacian, denoted by $D(\Delta)$, provided that there exists $h$ in $L^2$ such that
            \begin{equation}\label{eq:def_laplacian}
                -\int_{X}\langle \nabla f, \nabla g \rangle dm
                =
                \int_{X} h g dm, \quad \text{for any } g \in W^{1,2}.
            \end{equation}
            In this case, we denote $\Delta f=h$.
        \item \textbf{Measure\--Valued Laplacian:} We say that $f$ in $W^{1,2}_{loc}$ is in the domain of the measure-valued Laplacian, denoted by $D(\bm{\Delta})$, provided that there exists a signed Radon measure $\mu$ on $X$ such that,\footnote{Recall that $X$ is assumed to be proper, and therefore any bounded and closed set is compact on which Radon measures are finite.}
    \begin{equation}\label{eq:def_measured_laplacian}
        -\int_{X} \langle \nabla f, \nabla g \rangle dm
        =
        \int_{X} g d\mu, \quad \text{for any } g \in \mathrm{Lip}_{bs}.
    \end{equation}
    In this case, we denote $\bm{\Delta}f=\mu$. 
    \end{itemize}

\end{definition}
By the separating property of $\mathrm{Lip}_{bs}$ for Radon measures and the infinitesimal Hilbertian property, it's clear that both $\Delta$ and $\bm{\Delta}$ are well-defined, unique and linear operators.
Moreover, the two definitions are compatible in the following sense: on the one hand, if $f$ is in $W^{1,2}$ with $\bm{\Delta}f=\rho m$ for some $\rho$ in $L^2$, then $f$ is in $D(\Delta)$ and $\Delta f=\rho$.
On the other hand, if $f$ is in $W^{1,2}$ such that $\Delta f\in L^1$, then $f$ is in $D(\bm{\Delta})$ and $\bm{\Delta}f=(\Delta f) m$, see \citep[Proposition 6.2.13]{gigli2020lectures}.
For $f$ in the domain of $\bm{\Delta}$, we denote the Lebesgue decomposition with respect to $m$
\begin{equation}
    \bm{\Delta}f
    =
    (\bm{\Delta}^{ac}f )\cdot m + \bm{\Delta}^{s}f,
\end{equation}
where $\bm{\Delta}^{ac}f$ is the Radon\--Nikodym density and $\bm{\Delta}^sf$ is the singular part of $\bm{\Delta}f$.

For $w$ in $W^{1,2}\cap L^{\infty}$, we define the weighted Laplacian $\Delta_w$ similarly but with respect to the reference measure $m_w:=e^{w}\cdot m$ and test functions in $W^{1,2}(X,m_w)$.
For such $w$, it can be shown that $W^{1,2}$ coincides with $W^{1,2}(X,m_w)$, and the minimal relaxed gradient induced by $m_w$ coincides with the one induced by $m$, see \citep[Lemma 4.11]{ambrosio2014calculus}.
Moreover, it holds that $\Delta_w f=\Delta f+\langle \nabla w, \nabla f \rangle$, see \cite[Lemma 3.4]{gigli2020rigidity}

The following Lemma recaps the calculus rules, whose proofs can be found in \citep{gigli2020lectures, zhang2016local,gigli2013pde}.
\begin{lemma}
    Let $(X,d,m)$ be infinitesimally Hilbertian.
    Then:
    \begin{enumerate}[label=\textit{(\roman*)}]
        \item \textbf{Locality:} $|\nabla f|=|\nabla g|$ on $\{f-g=c\}$ for any $f$, $g$ in $W^{1,2}$ and constant $c$.
        \item \textbf{Chain rule:} for any $f$ in $W^{1,2}$ and Lipschitz function $\phi:\mathbb{R} \to \mathbb{R}$, it follows that
            \begin{equation*}
               |\nabla(\phi\circ f)|=|\phi'(f)||\nabla f|   
            \end{equation*}
            In particular, if $\phi$ is a contraction, then $|\nabla(\phi\circ f)|\leq |\nabla f|$.
        \item \textbf{Leibniz rule:} for any $f$, $g$ and $h$ in $W^{1,2}\cap L^{\infty}$, it follows that $fg$ is in $W^{1,2}$ and
            \begin{equation*}
                \langle \nabla(fg), \nabla h \rangle
                =
                f\langle \nabla g, \nabla h \rangle + g\langle \nabla f, \nabla h \rangle.
            \end{equation*} 
        \item \textbf{Chain rule:} for any $f$ in $D(\Delta)\cap \mathrm{Lip}_b$ and $C^2$-function $\phi:\mathbb{R}\rightarrow \mathbb{R}$, it follows that $\phi(f)$ is in $D(\Delta)$ and
            \begin{equation}\label{eq:Chain_rule}
                \Delta \phi (f)
                =
                \phi'(f) \Delta f + \phi''(f)|\nabla f|^2.
            \end{equation}
        \item \textbf{Leibniz rule:} for any $f$ and $g$ in $D(\bm{\Delta})\cap L^{\infty}$ such that $g$ is continuous and $\bm{\Delta}g$ is absolutely continuous with respect to $m$, then $fg$ is in $D(\bm{\Delta})$ and
            \begin{equation}\label{eq:Leibniz}
                \bm{\Delta}(fg)
                =
                f\bm{\Delta}g + g\bm{\Delta}f + 2\langle \nabla f, \nabla g \rangle\cdot m.
            \end{equation}
    \end{enumerate}
\end{lemma}
By the $L^2$-lower semicontinuity and convexity of the Cheeger energy, the heat semigroup $P_tf$ is defined as the gradient flow in $L^2$ of the Cheeger energy starting from $f\in L^2$ based on the classical Brezis\--Komura theory, which provides the existence and uniqueness results.
Moreover, for any $f\in L^2$, it holds that $t\mapsto P_t f$ is locally absolutely continuous on $(0,\infty)$ and $P_tf$ is in $D(\Delta)$ for all $t>0$ and
\begin{equation}
    \frac{d}{dt}P_t f = \Delta P_t f, \quad \text{for almost all } t\in (0,\infty).
\end{equation} 
Under the further assumption that $(X,d,m)$ is infinitesimally Hilbertian, the heat semigroup $P_t$ is linear, strongly continuous, contractive and order-preserving in $L^2$.
Moreover, $P_t$ can be extended into a linear, mass preserving and strongly continuous operator in $L^p$ for any $1\leq p<\infty$, see \citep{ambrosio2014calculus} and further results therein.

Finally, we recall the definition (see for example \cite[Definition 2.14]{honda2018local}) of the local Sobolev space $W^{1,2}(\Omega)$ on an open subset $\Omega\subset X$.
We denote by $\mathrm{Lip}_c(\Omega)$ the space of Lipschitz functions on $\Omega$ with compact support in $\Omega$ and by $\mathrm{Lip}_{loc}(\Omega)$ the space of locally Lipschitz functions on $\Omega$.
\begin{definition}[Local Sobolev space]
    Let $\Omega\subset X$ be an open subset.
    We say that $f\in L^2(\Omega)$ is in $W^{1,2}(\Omega)$ if
    \begin{enumerate}
        \item[(i)] $\eta f$ is in $W^{1,2}$ for all $\eta\in \mathrm{Lip}_c(\Omega)$; 
        \item[(ii)] $|\nabla f|\in L^2(\Omega)$ where $|\nabla f|:=|\nabla (\eta f)|$ $m$-a.e. on $\{\eta=1\}$ for $\eta\in \mathrm{Lip}_{c}(\Omega)$.
    \end{enumerate}
\end{definition}
We remark that by locality of the minimal relaxed gradient and Condition (i), Condition (ii) makes sense.
As for the Laplacian on $\Omega$, it is modified as follows: A function $f$ in $W^{1,2}(\Omega)$ belongs to $D(\Delta,\Omega)$ provided that there exists $g$ in $L^2(\Omega)$ such that
\begin{equation*}
    -\int_{\Omega} \langle \nabla f, \nabla \phi \rangle dm = \int_{\Omega} g \phi dm, \quad \text{for all }\phi \in \mathrm{Lip}_c(\Omega),
\end{equation*}
and we denote $\Delta_{\Omega}f:=g$.
Clearly, $\Delta_{\Omega}$ is linear by infinitesimal Hilbertianty and one can easily check that for $f$ in $D(\Delta)$, its restriction to $\Omega$ belongs to $D(\Delta,\Omega)$.

\subsection{RCD metric measure spaces}
\citet*{erbar2015equivalence,ambrosio2016bakry,ambrosio2019nonlinear} introduced the notion of the Riemannian curvature-dimension condition $\mathrm{RCD}^*(K,N)$ as the finite dimensional counterpart to $\mathrm{RCD}(K,\infty)$, itself introduced by \citet*{ambrosio2014metric} based on the curvature-dimension condition proposed by \citep{lott2009ricci,sturm2006geometry,sturm2006geometry2} to rule out the Finsler geometry.
It is shown in \cite{erbar2015equivalence,ambrosio2016bakry} that the $\mathrm{RCD}^*(K,N)$ spaces satisfy the so-called ``local\--to\--global'' property (see also \citep{bacher2010localization}).
Very recently, \citet*{cavalletti2021globalization} show that $\mathrm{RCD}(K,N)$ and $\mathrm{RCD}^*(K,N)$ conditions are equivalent if the reference measure is finite.

The notion of the $\mathrm{RCD}^*(K,N)$ condition can be defined in several equivalent ways, see \citep{erbar2015equivalence,ambrosio2016bakry,ambrosio2019nonlinear} for $\mathrm{RCD}^*(K,N)$ cases and \citep{ambrosio2014metric,ambrosio2015riemannian,ambrosio2015bakry} for $\mathrm{RCD}(K,\infty)$ cases.
In this work we give a definition including the case where $N=\infty$ from an Eulerian point of view based on the abstract $\Gamma$-calculus.
\begin{definition}
    We say that a metric measure space $(X,d,m)$ satisfies the $\mathrm{RCD}^*(K,N)$ condition for $K$ in $\mathbb{R}$ and $N$ in $[1,\infty]$ provided that
    \begin{enumerate}[label=\textit{(\roman*)}]
        \item $m(B(x,r))\leq C \exp(C r^2)$ for some $x$ in $X$ and $C>0$;
        \item \textbf{Sobolev-to-Lip property}: any $f$ in $W^{1,2}$ with $|\nabla f|$ in $L^{\infty}$ admits a Lipschitz representation $\tilde{f}$ in $\mathrm{Lip}$ such that $f=\tilde{f}$ $m$-a.e and $\mathrm{Lip}(f)=\||\nabla f|\|_{\infty}$.
        \item $(X,d,m)$ is infinitesimally Hilbertian.
        \item \textbf{Weak Bochner inequality}: for any $f$ and $g$ in $D(\Delta)$ with $\Delta f$ in $W^{1,2}$, $\Delta g$ in $L^\infty$ and $g \geq 0$, it holds
        \begin{equation}\label{eq:weak_Bochner}
            \int_{X}\Delta g \frac{|\nabla f|^2}{2}dm
            \geq
            \int_{X}g \left( \frac{1}{N}(\Delta f)^2 + \langle \nabla f, \nabla \Delta f \rangle + K |\nabla f|^2 \right)dm.
        \end{equation}
    \end{enumerate}
\end{definition}
In the following, we assume that $(X,d,m)$ is a compact $\mathrm{RCD}^\ast(K,N)$ space with $N$ in $(1, \infty)$, $K\in \mathbb{R}$ and $\mathrm{diam}(X)=D<\infty$, which is the framework of the results in this work

First, $X$ satisfies the generalized Bishop\--Gromov inequality, that is, for any $0<r<R$ and $x\in X$,
\begin{equation}
    \frac{m(B(x,R))}{v_{K,N}(R)}
    \leq
    \frac{m(B(x,r))}{v_{K,N}(r)},
\end{equation}
where $v_{K,N}(r)$ is the volume of ball with radius $r>0$ in the model space (see \cite[Theorem 6.2]{bacher2010localization}).
In particular, $X$ is globally doubling with the constant $2^N$ when $K\geq 0$, that is,
\begin{equation}
    m(B(x,2r)) \leq 2^N m(B(x,r)),\quad \text{for any } x\in X \text{ and } r>0,
\end{equation}
and with the constant $C(K,N,D)$ depending only on $K,N$ and $D$ when $K<0$.
It also holds that $m(B(x,r))>0$ for any $r>0$ and $x$ in $X$ and that $m(X)<\infty$.
Thus by \cite{cavalletti2021globalization}, $X$ also satisfies the $\mathrm{RCD}(K,N)$ condition.

Second, $X$ supports the weak $(1,1)$-Poincaré inequality, see \citep[Theorem 1.1]{rajala2012interpolated}, that is, for any $x$ in $X$, $r>0$ and any continuous function $f:X\rightarrow \mathbb{R}$ and any upper gradient $g$ of $f$, we have
\begin{equation}\label{eq:Poincare_ineq}
    \fint_{B(x,r)}\left| f- (f)_{x,r} \right|dm
    \leq
    C r \fint_{B(x,2r)}g dm,
\end{equation}
where the constant $C$ only depends on $K$, $N$ and $r$, and $\fint_{\Omega}f dm:=\frac{1}{m(\Omega)}\int_{\Omega}f dm$ and $(f)_{x,r}:=\int_{B(x,r)}f dm$.
Clearly, by H\"older inequality, $X$ also supports the weak $(1,q)$-Poincaré inequality for any $q\in (1,\infty)$.
This implies that $X$ is connected, see \citep[Proposition 4.2]{bjorn2011nonlinear}.
In fact, one can show that $X$ is also a geodesic space.

Third, when $K>0$, the Bonnet-Myers theorem implies that $\mathrm{diam}(X)\leq \pi\sqrt{(N-1)/K}$, see \citep[Corollary 2.6]{sturm2006geometry2}.
Since the normalization of the reference measure does not affect $\mathrm{RCD}^*(K,N)$ conditions, it is not restrictive to assume that $m(X)=1$ when $X$ is compact.

Fourth, in the $\mathrm{RCD}$ setting, for Sobolev functions $f\in W^{1,2}$, it is possible to identify the gradient $\nabla f$, rather than the modulus of the gradient $|\nabla f|$, as the unique element in the tangent module $L^2(TX)$, which is a $L^2(m)$-normed $L^{\infty}(m)$-module, see \citep{gigli2020lectures,gigli2018nonsmooth}.
Therein, a second-order calculus on $\mathrm{RCD}(K,\infty)$ spaces is also developed such that the notions of Hessian $\mathrm{Hess}f$ and its pointwise norm $|\mathrm{Hess}f|_{HS}\in L^2$ are well-defined.
For the complete theory, we refer readers to \citep{gigli2020lectures,gigli2018nonsmooth}.
Here we only mention the inclusion $D(\Delta)\subseteq H^{2,2}$.
Consequently, for any $f$ in $D(\Delta)\cap \mathrm{Lip}$, we have $|\nabla f|^2$ in $W^{1,2}$ and
\begin{equation}\label{eq:H_2_2}
    |\nabla |\nabla f|^2|
    \leq
    2 |\mathrm{Hess}f|_{HS}|\nabla f|,
\end{equation}
see \citep[Theorem 3.3.18]{gigli2018nonsmooth} or \citep[Lemma 3.5]{debin2021quasi}.
\begin{remark}
    Although there exist different notions of the Sobolev space such as the Newtonian space and the Cheeger\--Sobolev space, see \citep{bjorn2011nonlinear,shanmugalingam2000newtonian,cheeger1999}, and different notions of weak upper gradients on metric measure spaces, all these notions are equivalent to each other in the setting of $\mathrm{RCD}^*(K,N)$ spaces.
    In particular, the minimal relaxed gradient coincides with the minimal weak upper gradient, which is defined via test plans and geodesics, see \citep[Definition 2.1.8]{gigli2020lectures}, and $W^{1,2}$ is reflexive.
    For more details, we refer to \citep{ambrosio2013density}.
\end{remark}

Fifth, we recall a regularity result of the Poisson equation on the ball, see \citep[Lemma 3.4]{zhang2016local} and \citep[Theorem 3.1]{koskela2003lipschitz} proved in the setting of the heat semigroup  curvature condition.
\begin{lemma}\label{lemma:Poisson_regularity_Zhu}
    Let $(X,d,m)$ be a $\mathrm{RCD}^*(K,N)$ space with $K$ in $\mathbb{R}$ and $N$ in $[1,\infty)$.
    Let further $g$ be in $L^{\infty}(B_R)$ where $B_R:=B(x_0,R)$ is a geodesic ball centered at some $x_0$ in $X$ with radius $R>0$.
    Assume that $f$ is in $W^{1,2}(B_R)$ and $\Delta_{B_R}f=g$ on $B_R$.
    Then it holds that $|\nabla f|$ is in $L^{\infty}_{loc}(B_R)$ and
    \begin{equation*}
        \| |\nabla f|\|_{L^{\infty}(B_{R/2})}
        \leq
        C(N,K,R)\left(\frac{1}{m(B_R)}\|f\|_{L^1(B_R)} + \|g\|_{L^{\infty}(B_R)}\right).
    \end{equation*}
\end{lemma}
\begin{remark}
    Note that our definition of the local Sobolev space $W^{1,2}(\Omega)$ on some open domain $\Omega$ is a priori different from the one $H^{1,2}(\Omega)$ used in \citep{koskela2003lipschitz}, which is the completion of $\mathrm{Lip}_{loc}(\Omega)$ with respect to the norm $\|f\|_{H^{1,2}(\Omega)}:=\|f\|_{L^2(\Omega)}+\| \mathrm{lip}(f)\|_{L^2(\Omega)}$.
    However,  on the one hand, as shown in \cite[Remark 2.15]{honda2018local}, this space coincides with the Cheeger\--Sobolev space $Ch^{1,2}(\Omega)$ defined by \citet*{cheeger1999} on $\mathrm{RCD}^*(K,N)$ spaces.
    On the other hand, since $\mathrm{Lip}_{loc}(\Omega)$ is dense in $Ch^{1,2}(\Omega)$ (see \cite[Theorem 5.47]{bjorn2011nonlinear}) and the norms of $Ch^{1,2}(\Omega)$ and $H^{1,2}(\Omega)$ are equal for $\mathrm{Lip}_{loc}(\Omega)$ on $\mathrm{RCD}^*(K,N)$ spaces\footnote{Since the local Lipschitz constant of Lipschitz functions coincides with their minimal weak upper gradient in the Cheeger sense, see \cite[Theorem A.7]{bjorn2011nonlinear}}, these two spaces coincide with each other.
    Consequently, $H^{1,2}(\Omega)$ and $W^{1,2}(\Omega)$ coincide with each other.
    
\end{remark}

We also recall the following result of the Sobolev inequality for the compact $\mathrm{RCD}(K,N)$ space with $K\in \mathbb{R}$ and $N$ in $(2,\infty)$, proved in \cite{nobili2022rigidity}.
As a result, we have a metric version of the Rellich\--Kondrachov type theorem for compact $\mathrm{CD}(K,N)$ spaces with $K\in \mathbb{R}$ and $N\in (2,\infty)$.
Similar results have been previously proved by \citet*{profeta2015sharp} for $\mathrm{RCD}^*(K,N)$ spaces with $K>0$ and $N\in (2,\infty)$.
\begin{lemma}\label{lemma:Pre_Sobolev_embedding}
    Let $(X,d,m)$ be a compact $\mathrm{RCD}^*(K,N)$ space with $K\in \mathbb{R}$ and $N$ in $(2,\infty)$ and $\mathrm{diam}(X)\leq D$ for some $D>0$ and $m(X)=1$.
    \begin{enumerate}[label=\textit{(\roman*)}, fullwidth]
        \item \textbf{Sobolev inequality (\cite[Proposition 5.1]{nobili2022rigidity})}:
            there exist a constant $A>0$ depending only on $K,N,D$, such that for every $f$ in $W^{1,2}$, it holds
            \begin{equation}\label{eq:Sobolev_ineq}
                \| f \|^2_{2^*}
                \leq
                \|f\|^2_2 + A\cdot \mathrm{Ch}(f),
            \end{equation}
            where $2^*=2N/(N-2)$ is the Sobolev conjugate of $2$.
        \item\label{lemma:Rellich_Kondrachov} \textbf{Rellich\--Kondrachov}: let $(f_n)$ be a sequence in $W^{1,2}$ with $\sup_{n}\|f_n\|_{W^{1,2}}<\infty$.
            Then there exists $f$ in $W^{1,2}$ and a subsequence $(f_{n_k})_k$ such that for every $1\leq q<2^*$, it holds that
            \begin{equation*}
                f_{n_k}\rightarrow f \quad \text{in } L^q(X,m).
            \end{equation*}
    \end{enumerate}
\end{lemma}
The proof of the Rellich\--Kondrachov type theorem above follows the argument in \cite[Theorem 8.1]{hajlasz2000sobolev} and the equivalent characterizations of weak Poincaré inequalities from \citet*{keith2003modulus}, which is actually the same as the one in \cite[Proposition 4.2]{profeta2015sharp} for $\mathrm{RCD}^*(K,N)$ spaces with $K>0$ and $N\in (2,\infty)$.
For the sake of completeness, we give the proof in the Appendix \ref{appendix_B}.

Finally, we mention a key result about the heat semigroup and the resolvent of the Laplacian.
Recall that for the heat semigroup $P_t$, we say that $P_t$ is ultracontractive if for $1\leq p<q\leq \infty$, there exists a constant $C(t)>0$ such that for any $f$ in $L^p$, it holds that
\begin{equation*}
    \| P_t f\|_{q} \leq C(t)\|f\|_{p},\quad t>0,
\end{equation*}
and we denote $\|P_t\|_{(p,q)}:=\sup_{\|f\|_{p}\leq 1}\|P_t f\|_{q}$.
The ultracontractive property of the heat semigroup is equivalent to the Sobolev inequality for the Markov triple associated with the heat semigroup, see \citep[Theorem 6.3.1]{bakry2014analysis}.
For compact $\mathrm{RCD}^{*}(K,N)$ spaces with $K\in \mathbb{R}$ and $N$ in $(2,\infty)$, since the Sobolev inequality holds, it follows that $P_t$ has the ultracontractive property.
More precisely, for $1\leq p<q\leq \infty$ and $0<t\leq 1$, it holds that
\begin{equation*}
    \| P_t \|_{(p,q)} \leq C t^{-\frac{N}{2}(\frac{1}{p}-\frac{1}{q})},
\end{equation*}
where $C>0$ is a constant depending on $K$ and $N$.
As a consequence, the ultracontractive property of the heat semigroup provides the following boundedness result for the resolvent operator of the Laplacian $R_{\lambda}:=(\lambda I - \Delta)^{-1}$ for $\lambda>0$, the proof of which can be found in \citep[Lemma 4.1]{profeta2015sharp} and \citep[Corollary 6.3.3]{bakry2014analysis}.
\begin{lemma}\label{lemma:ultracontra_resolvent}
    On a compact $\mathrm{RCD}^*(K,N)$ space with $K\in \mathbb{R}$ and $N$ in $(2,\infty)$, let $\lambda>0$.
    If $1\leq p\leq N/2$, the resolvent $R_{\lambda}:L^p\rightarrow L^q$ is bounded for each $1\leq q < pN/(N-2p)$.
    If $p>N/2$, the resolvent $R_{\lambda}: L^p \rightarrow L^{\infty}$ is bounded.
\end{lemma}

\section{Existence of positive extremal functions}
From now on, we assume that $(X,d,m)$ is a compact $\mathrm{RCD}^*(K,N)$ space with $K\in \mathbb{R}$ and $N$ in $(2,\infty)$, and $m$ is a Borel probability measure with full support.
We consider the following variational problem
\begin{equation}\label{eq:var_pbm_general}
    \inf\left\{\int_{X}\left(|\nabla f|^2 +\alpha_1 f^2 - \frac{\alpha_2}{2} f^2 \log f^2\right)dm : f\in W^{1,2},\|f\|_2=1 \right\},
\end{equation}
where $\alpha_1$ and $\alpha_2$ are in $\mathbb{R}$.
The infimum quantity of variational problem \eqref{eq:var_pbm_general} is called the log Sobolev constant $\lambda:=\lambda(\alpha_1,\alpha_2)$ on $X$ with parameters $\alpha_1$ and $\alpha_2$, and the functional in \eqref{eq:var_pbm_general} is called the log-Sobolev functional.
\begin{definition}\label{def:extremal_function}
    Provided that $\lambda=\lambda(\alpha_1,\alpha_2)$ is a finite number, we call a function $u$ in $W^{1,2}$ the extremal function of the variational problem \eqref{eq:var_pbm_general} if $\|u\|_2=1$ and
    \begin{equation}
        \lambda
        =
        \int_{X}\left( |\nabla u|^2 +\alpha_1 u^2 - \frac{\alpha_2}{2} u^2 \log u^2\right)dm.
    \end{equation} 
\end{definition}

In the following, we provide our main result of this section.
It states the existence of non-negative extremal functions of the variational problem \eqref{eq:var_pbm_general}.
Moreover, we show that all non-negative extremal functions are actually Lipschitz continuous and bounded away from zero on $X$.
As a corollary, we show that the logarithmic transform of any non-negative extremal functions is Lipschitz and in the domain of the Laplacian and satisfies some Euler\--Lagrange equation.
\begin{theorem}\label{thm:var_general}
    Let $(X,d,m)$ be a compact $\mathrm{RCD}^*(K,N)$ spaces with $K\in \mathbb{R}$ and $N\in (2,\infty)$, and let $\alpha_1$ and $\alpha_2$ be given constants with $\alpha_2>0$.
    Then the log Sobolev constant $\lambda=\lambda(\alpha_1,\alpha_2)$ has finite value and the variational problem \eqref{eq:var_pbm_general} admits non-negative extremal functions.
    Moreover, any non-negative extremal function $u$ satisfies the following properties:
    \begin{enumerate}[label=\textit{(\roman*)}]
        \item $u$ is in $D(\Delta)$ and satisfies
            \begin{equation}\label{eq:var_pbm_pde}
                -\Delta u = \alpha_2 u\log u + (\lambda - \alpha_1)u.
            \end{equation}
            Furthermore, if $\lambda\neq \alpha_1$, then $u$ is non-constant.
        \item $u$ is Lipschitz continuous;
        \item $u$ is positive.\footnote{Since $X$ is compact, $u>\delta >0$ for some $\delta$.}
    \end{enumerate}
\end{theorem}

\begin{corollary}\label{corollary:var_log_transform}
    Let $\alpha_1$ and $\alpha_2$ be given constants with $\alpha_2>0$ and $u$ be an arbitrary non-negative extremal function of \eqref{eq:var_pbm_general}.
    Then $v:=\log u$ is Lipschitz and in $D(\Delta)$ and satisfies
    \begin{equation}\label{eq:var_log_transform_pde}
        -\Delta v = |\nabla v|^2 +\alpha_2v + \lambda-\alpha_1.
    \end{equation}
    In particular, the equation $-\Delta f=|\nabla f|^2 + \alpha_2 f $ admits Lipschitz weak solutions which are non-constant whenever $\lambda\neq \alpha_1$. 
\end{corollary}
\begin{remark}
    A classical example of the above variational problem is the weak log-Sobolev inequalities, that is, the variational problems
    \begin{equation}
        \lambda_{\varepsilon}
        =
        \inf\left\{\frac{2}{\alpha_{\varepsilon}}\int_{X}|\nabla f|^2dm + \int_{X}\left(\varepsilon f^2-f^2\log f^2\right)dm: f\in W^{1,2}, \|f\|_2=1 \right\},
    \end{equation}
    where $\varepsilon\geq0$ and $\alpha_{\varepsilon}$ is defined as the supremum of those $C>0$ such that
    \begin{equation*}
        \int_{X}f^2\log f^2dm\leq \int_{X}\left(\frac{2}{C}|\nabla f|^2 + \varepsilon f^2\right) dm
    \end{equation*}
    for all $f$ in $W^{1,2}$ and $\| f\|_2 = 1$.

    The log-Sobolev inequality on the $\mathrm{RCD}^*(K,N)$ space with $K>0$ and $N\in (1,\infty)$, see \citep{cavalletti2017sharp}, implies that $\alpha_{\varepsilon}\geq KN/(N-1)$ for all $\varepsilon\geq 0$, and $\alpha_{\varepsilon}<\infty$ when $\varepsilon$ is small enough.
    A straightforward inspection shows that if $\alpha_{\varepsilon}$ has finite value, then $\lambda_{\varepsilon}=0$.
    In this case, by using Theorem \ref{thm:var_general} and Corollary \ref{corollary:var_log_transform}, we obtain that the weak Sobolev inequality admits Lipschitz and positive extremal functions which are further non-constant if $\varepsilon > 0$.
    Furthermore, we derive that the equation $-\Delta f=|\nabla f|^2 +\alpha_{\varepsilon}f$ admits non-constant, Lipschitz and positive weak solutions if $\varepsilon> 0$.
\end{remark}

We first prove all assertions of Theorem \ref{thm:var_general} but positivity.
While the existence result follows classical methods in \citep{rothaus1981logarithmic} using variational techniques together with the compact Sobolev embedding, the proofs of the boundedness and Lipschitz regularity are different.
Our methods are based on the ultracontractive property of the resolvent and the local regularity result of the Poisson equation on a ball in Lemma \ref{lemma:Poisson_regularity_Zhu}.
In the following proofs, we denote by $C>0$ a universal constant which may vary from one step to the next.
\begin{proof}[First part of the proof of Theorem \ref{thm:var_general}]
    \begin{enumerate}[label=\textbf{\textsc{Step \arabic*:}}, fullwidth]
        \item We show the existence of non-negative extremal functions of variational problem \eqref{eq:var_pbm_general}.
            Let $\mathcal{F}:=\{f\in W^{1,2}: \|f\|_2=1\}$ and $F:\mathcal{F}\rightarrow \mathbb{R}$ be the log Sobolev functional in \eqref{eq:var_pbm_general}, that is,
            \begin{equation}
                F(f):=F_{\alpha_1,\alpha_2}(f)
                =
                \int_{X}\left(|\nabla f|^2 + \alpha_1 f^2 - \frac{\alpha_2}{2} f^2\log f^2\right) dm, \quad \text{for }f\in \mathcal{F}.
            \end{equation}
            We claim that $F$ is coercive on $\mathcal{F}$.
            Indeed, choosing $0<\delta<N/(N-2)$, since $\|f\|_2=1$, by Jensen's inequality, it follows that
            \begin{equation}
                \int f^2 \log f^2 dm
                =
                \frac{1}{\delta}\int \log |f|^{2\delta}(f^2 dm)
                \leq
                \frac{1+\delta}{\delta}\log \|f\|^2_{2+2\delta}.
            \end{equation}
            Note that the Sobolev inequality, see Lemma \ref{lemma:Pre_Sobolev_embedding}, implies that
            \begin{equation}
                \|f\|^2_{2+2\delta}
                \leq
                \|f\|^2_{2^*}
                \leq
                \|f\|^2_2 + A\cdot\mathrm{Ch}(f)
                \leq
                C(1+\mathrm{Ch}(f)),
            \end{equation}
            and therefore
            \begin{equation}
                F(f)
                \geq
                2\mathrm{Ch}(f) + \alpha_1 - \frac{(1+\delta)\alpha_2}{2\delta}\log C\left(1+ \mathrm{Ch}(f)\right),
            \end{equation}
            which implies that $F\rightarrow \infty$ for $f$ in $\mathcal{F}$ with $\|f\|_{W^{1,2}}\rightarrow \infty$.
            Let $(f_n)\subseteq \mathcal{F}$ be a minimizing sequence which can be assumed non-negative since $\mathrm{Ch}(|f|)\leq \mathrm{Ch}(f)$.
            By the coercivity of $F$, it follows that $(f_n)$ is bounded in $W^{1,2}$.
            The compact Sobolev embedding implies the existence of a non\--negative $u$ in $W^{1,2}\cap \mathcal{F}$ and a subsequence of $(f_n)$, relabelled as $(f_n)$, such that $f_n\rightarrow u$ strongly in $L^q$ for all $q$ in $[1,2^*)$.
            By the $L^2$-lower semicontinuity of the Cheeger energy as well as the $L^{2+\delta}$-continuity of $\int f^2\log f^2$ for small $\delta>0$,\footnote{Which can be shown from mean value theorem and equality $f^2\log f^2 -g^2\log g^2=2(|f|-|g|)(1+\log \theta^2)\theta$ where $\theta$ is between $|f|$ and $|g|$, see \textcolor{blue}{\citep[page 112]{rothaus1981logarithmic}}} it follows that $F$ reaches its minimum at $u$.
            Furthermore, if $\alpha_1\neq \lambda$, a straightforward inspection shows that $u\equiv 1$ cannot be an extremal function.

        \item Given any non-negative extremal function $u$, we show that $u\in D(\Delta)$ and satisfies \eqref{eq:var_pbm_pde} by using the classical Euler\--Lagrangian method.
            Let $\phi\in \mathrm{Lip}_{bs}$ be an arbitrary Lipschitz function and define the functional $G$ as follows
            \begin{equation}
                G(t,\beta)
                =
                F(u+t\phi) + \beta \left(\|u+t\phi\|_2^2 -1\right).
            \end{equation}
            The mean value theorem yields
            \begin{equation}
                \left|u^2\log u^2 - (u+t\phi)^2\log (u+t\phi)^2\right|
                \leq
                2t|\phi|\left|(1+\log \theta^2)\theta \right|
            \end{equation}
            with a function $\theta$ taking values between $|u|$ and $|u+t\phi|$.
            From the inequality $|\theta\log \theta|\leq \max(1/e, C(\delta)\theta^{1+\delta})$ for some small $\delta>0$ together with the dominated convergence of Lebesgue, it follows that
            \begin{equation}
                0
                =
                \frac{1}{2}\frac{dG(t, \beta)}{dt}\Big|_{t=0}
                =
                \int \left(\langle \nabla u, \nabla \phi \rangle  + \alpha_1 u \phi -  \frac{\alpha_2}{2} u\phi \log u^2 - \alpha_2 u\phi + \beta u \phi \right) dm. 
            \end{equation}
            Since $\mathrm{Lip}_{bs}$ is dense in $W^{1,2}$ and $u\log u$ is in $L^2$ by the Sobolev inequality, it follows that
            \begin{equation}\label{eq:var_pbm_2}
                \int \langle \nabla u, \nabla \phi \rangle dm
                =
                \int \left(\frac{\alpha_2}{2} \phi u \log u^2  + (\alpha_2-\alpha_1-\beta)\phi u\right) dm, \quad \text{for all } \phi \in W^{1,2}.
            \end{equation}
            Plugging $\phi = u$ into \eqref{eq:var_pbm_2} and using the real value of the log Sobolev constant $\lambda$, it follows that $\alpha_2-\beta=\lambda$.
            By the definition of the Laplacian we deduce that $u$ is in $D(\Delta)$ and that $-\Delta u=\alpha_2 u\log u +(\lambda-\alpha_1)u$.
        \item We show that $u$ is in $\mathrm{Lip}$.
            We start by showing that $u$ is in $L^{\infty}$.
            Let $\beta>0$. 
            Note that using the resolvent of the Laplacian, \eqref{eq:var_pbm_pde} can be rewritten as
            \begin{equation}\label{eq:var_pbm_3_resolvent}
                u
                =
                \left(\beta I - \Delta\right)^{-1}\left( \alpha_2u\log u + (\beta+\lambda -\alpha_1)u\right)
                =
                R_{\beta}\left(\alpha_2u\log u + \alpha_3 u\right),
            \end{equation}
            where $I$ is the identity map and $R_{\beta}$ is the resolvent of the Laplacian and $\alpha_3:=\beta+\lambda-\alpha_1$.
            Note that in order to prove that $u$ is in $L^{\infty}$, it is sufficient to show that $g:=\alpha_2u\log u + \alpha_3u$ is in $L^r$ for some $r>N/2$, according to Lemma \ref{lemma:ultracontra_resolvent}.
            By the Sobolev inequality, we know that $u$ is in $L^{2^*}$, which implies that $g$ is in $L^{r}$ for all $r\in (2,2^*)$, (recall that $2^\ast = 2N/(N-2)$).
            Fix $\delta\geq 2N$.
            It implies that $2<r_1<2^*$ for $1/r_1=1/2^*+1/\delta$.
            If $r_1>N/2$, then Lemma \ref{lemma:ultracontra_resolvent} and the identity \eqref{eq:var_pbm_3_resolvent} implies that $u\in L^{\infty}$.
            If $r_1\leq N/2$, then Lemma \ref{lemma:ultracontra_resolvent} implies that $u\in L^r$ for all $r<r_1N/(N-2r_1)$.
            Repeating the previous step, we can obtain that $g\in L^{r_2}$ for some $r_2<r_1N/(N-2r_1)$ such that
            \begin{equation}\label{eq:var_pbm_add_1}
                \frac{1}{r_2}
                =
                \frac{1}{r_1} - \frac{2}{N} + \frac{1}{\delta}
                =
                \frac{1}{2^*}-\frac{2}{N} + \frac{2}{\delta}.
            \end{equation}
            Define by induction $1/r_k:=1/2^*-2(k-1)/N+k/\delta$ if $r_{k-1}\leq N/2$ which by definition of $\delta$ implies that after finitely many iterations $g$ is in $L^{r_k}$ for some $r_k>N/2$ and deduce that $u$ is in $L^{\infty}$.

            We now show that $u$ is actually in $\mathrm{Lip}$.
            Let $x_0$ be any point in $X$ and $B_R=B(x_0, R)$ be an open ball with $0<R<\mathrm{diam}(X)/3$ and $g:=\alpha_2 u\log u + (\lambda-\alpha_1)u$.
            Since $u$ is in $D(\Delta)$ and the identity \eqref{eq:var_pbm_pde}, it follows by the definition that the restriction of $u$ on $B_R$, which we still denote by $u$, belongs to $W^{1,2}(B_R)$ as well as $D(\Delta,B_R)$ and that $\Delta_{B_R}u=-g$ holds on $B_R$ in the distributional sense.
            Hence, by Lemma \ref{lemma:Poisson_regularity_Zhu}, it follows that $|\nabla u|\in L^{\infty}_{loc}(B_R)$ and
            \begin{equation}
                \| |\nabla u|\|_{L^{\infty}(B_{R/2})}
                \leq
                C(N,K,R)\left(\frac{1}{m(B_R)}\|u\|_{L^2} + \|g\|_{L^{\infty}}\right).
            \end{equation}
            However, $X$ is compact and $m(B_R)>0$ by the doubling property, implying that $|\nabla u|$ belongs to $L^{\infty}$.
            The Sobolev-to-Lipschitz property thus implies that $u$ has a Lipschitz representation ending the proof of Theorem \ref{thm:var_general} but positivity.
    \end{enumerate}
\end{proof}
As for the last assertion of Theorem \ref{thm:var_general}, the positivity of non-negative extremal functions, we address the following auxiliary lemma stating that any non-negative extremal function vanishing at one point must also vanish in a neighborhood of that point.
Our approach is based on a maximum principle type result for the De Giorgi class proved in \citep{kinnunen2001regularity}.
\begin{lemma}\label{lemma:strict_positive}
    Suppose that the hypothesis of Theorem \ref{thm:var_general} holds.
    Let $u$ be any non-negative extremal function of variational problem \eqref{eq:var_pbm_general}.
    Assume that $u(x_0)=0$ for some $x_0$ in $X$, then $u\equiv 0$ on a neighborhood of $x_0$.
\end{lemma} 
\begin{proof}
    From the first part of the proof of Theorem \ref{thm:var_general}, any non-negative extremal function of \eqref{thm:var_general} is Lipschitz continuous.
    Furthermore, since $\alpha_2 >0$ and $\lambda$ is finite, the function $g(t):=\alpha_2 t\log t + (\lambda-\alpha_1)t\leq 0$ for all small enough $t>0$.
    Hence we can find $r_0>0$ such that $g(u)\leq 0$ on $B(x_0,r_0)$.
    We first claim that $-u$ is of De Giorgi class $\mathrm{DG}_2(B(x_0,r_0))$.
    In other terms, there exists $C>0$ such that for all $k$ in $\mathbb{R}$ and $z$ in $B(x_0,r_0)$ and all $0<\rho< R\leq \mathrm{diam}(X)/3$ with $B(z,R)\subseteq B(x_0,r_0)$, it holds that
    \begin{equation}
        \int_{A_z(k,\rho)}|\nabla u|^2 dm
        \leq \frac{C}{(R-\rho)^2}\int_{A_z(k,R)}(-u-k)^2dm,
    \end{equation}
    where $A_z(k,r):=\{x\in B(z,r): -u(x)>k\}$.
    In the following, $C$ denotes a positive constant, which is independent of the choice of $z,\rho, r, R, k$, and may vary from line to line.
    Let $\eta$ be a Lipschitz cut-off function such that $\eta=1$ on $B(z,\rho)$ and $\mathrm{supp}(\eta)\subseteq B(z,R)$ with $|\nabla \eta|\leq C/(R-\rho)$ for some $C>0$.
    Taking $\phi=\eta (-u-k)_{+}$ as test function for \eqref{eq:var_pbm_pde}, it follows that
    \begin{equation}\label{eq:var_pos_1}
        \int \langle \nabla(\eta(-u-k)_+), \nabla u \rangle dm
        =
        \int \eta (-u-k)_{+}\left(\alpha_2 u\log u + (\lambda-\alpha_1)u\right)dm.
    \end{equation}
    As for the left-hand side of \eqref{eq:var_pos_1}, the Leibniz rule yields
    \begin{multline}\label{eq:var_pos_2}
        \int \langle \nabla(\eta(-u-k)_{+}), \nabla u \rangle dm
        \geq
        \int_{\{-u>k\}} \eta |\nabla u|^2 dm - \int_{\{-u>k\}} (-u-k)_{+} |\nabla \eta||\nabla u | dm\\
        \geq
        \int_{A_z(k,\rho)}|\nabla u|^2 dm - \frac{C}{(R-\rho)^2}\int_{A_z(k,R)}(-u-k)^2_{+}dm - \int_{A_z(k,R)\setminus A_z(k,\rho)}C|\nabla u|^2dm,
    \end{multline}
    where we use the locality of the minimal weak upper gradient $|\nabla(-u-k)|=|\nabla u|$ for the first inequality, and Young's inequality together with $|\nabla \eta|\leq C/(R-\rho)$ and $|\nabla \eta|=0$ on $B(z,\rho)$ for the second inequality.
    As for the right hand side of \eqref{eq:var_pos_1}, by the very choice of $B(x_0,r_0)$ and the non-negativity of $u$, it follows that $\alpha_2 u\log u + (\lambda-\alpha_1)u\leq 0$.
    Hence
    \begin{equation}\label{eq:var_pos_3}
        \int \eta(-u-k)_{+}\left(\alpha_2u\log u + (\lambda-\alpha_1)u\right) dm
        \leq 0.
    \end{equation}
    With \eqref{eq:var_pos_2} and \eqref{eq:var_pos_3}, we obtain that
    \begin{equation}
        \int_{A_z(k,\rho)}|\nabla u|^2 dm
        \leq
        \frac{C}{(R-\rho)^2}\int_{A_z(k,R)} (-u-k)^2_{+}dm + C\int_{A_z(k,R)\setminus A_z(k,\rho)}|\nabla u|^2 dm.
    \end{equation}
    Adding $C\int_{A_z(k,\rho)}|\nabla u|^2dm$ and then dividing by $(1+C)$ on both sides, it follows that
    \begin{equation}
        \int_{A_z(k,\rho)}|\nabla u|^2dm
        \leq
        \frac{C}{(R-\rho)^2}\int_{A_z(k,R)}(-u-k)^2_{+}dm + \theta \int_{A_z(k,R)}|\nabla u|^2 dm,
    \end{equation}
    where $\theta=C/(1+C)\in (0,1)$.
    Applying the same argument to $0<\rho<r\leq R$ and enlarging the domain of integral, we derive that for all $0<\rho<r\leq R$,
    \begin{equation}
        \int_{A_z(k,\rho)}|\nabla u|^2dm
        \leq
        \frac{C}{(r-\rho)^2}\int_{A_z(k,R)}(-u-k)^2_{+}dm + \theta \int_{A_z(k,r)}|\nabla u|^2 dm.
    \end{equation}
    By \citep[Lemma 3.2]{kinnunen2001regularity} (see also \citep[Lemma 3.1 in p. 161]{mariano1983}) with $f(r):=\int_{A_z(k,r)}|\nabla u|^2dm$, we obtain that
    \begin{equation}
        \int_{A_{z}(k,\rho)}|\nabla u|^2 dm
        \leq
        \frac{C}{(R-\rho)^2}\int_{A_z(k,R)}(-u-k)^2_{+}dm,
    \end{equation}
    which implies that $-u$ is of De Giorgi class $\mathrm{DG}_2(B(x_0,r_0))$.

    Note that by the assumption that $X$ is a compact $\mathrm{RCD}^*(K,N)$ space with $K\in \mathbb{R}$ and $N\in (2,\infty)$, it follows that $X$ is global doubling and supports the global weak $(1,1)$-Poincaré inequality, see \citep{rajala2012interpolated}.
    Together with Hölder's inequality, $X$ supports the global weak $(1,q)$-Poincare inequality for any $q$ in $(1,2)$.
    Since the minimal weak upper gradient in the $\mathrm{RCD}$ setting coincides with the minimal weak upper gradient in the Newtonian setting, see \cite{ambrosio2013density}, it follows that the $\mathrm{RCD}^*(K,N)$ space together with that $u\geq 0$ and $-u\in \mathrm{DG}_2(B(x_0,r_0))$ satisfies the assumptions of \citep[Lemma 6.1 and Lemma 6.2]{kinnunen2001regularity}.
    
    We now prove the assertion by contradiction. 
    Suppose that the assertion of our Lemma does not hold.
    Let then $0<R<r_0$ and $x$ in $B(x_0,R)$ and $\tau:=u(x)>0$ be fixed.
    By the Lipschitz continuity of $u$, we can find $0<r\leq R$ with $B(x,r)\subseteq B(x_0,R)$ such that $u\geq\tau/2$ on $B(x, r)$.
    The generalized Bishop\--Gromov inequality yields\footnote{For this fact, see \cite[Theorem 2.14]{nobili2022rigidity} or \cite[Corollary 30.12]{Villani_book_2008}}
    \begin{equation}
        \frac{m(B(x,r))}{m(B(x_0,R))}
        \geq
        C(K,N,R)\left(\frac{r}{R}\right)^N,
    \end{equation}
    for some constant $C(K,N,R)>0$ depending only on $K,N,R$.
    Taking $r$ even smaller if necessary, we can assume that $0<C(K,N,R)(r/R)^N<1$.
    Hence it follows that
    \begin{equation}\label{eq:var_pos_final}
        m\left(\left\{z\in B(x_0,R):  u(z)\geq \tau/2\right\}\right)
        \geq
        m\left(B(x,r)\right)
        \geq
        C(K,N,R)\left(\frac{r}{R}\right)^N m(B(x_0,R)).
    \end{equation}
    Taking $\gamma:=1-C(K,N,R)(r/R)^N$, the inequality \eqref{eq:var_pos_final} implies that
    \begin{equation}
        m\left(\left\{x\in B(x_0,R):  u< \tau/2\right\}\right)
        \leq
        \gamma \cdot m(B(x_0,R)).
    \end{equation}
    Since $0<\gamma <1$, \citep[Lemma 6.2]{kinnunen2001regularity} yields the existence of $\bar{\lambda}=\bar{\lambda}(\gamma)>0$ such that
    \begin{equation}
        u(x_0)=\inf_{B(x_0, R/2)}u \geq \frac{\bar{\lambda} \tau}{2}>0,
    \end{equation}
    which is a contradiction.
\end{proof}
We can now address the positivity assertion of Theorem \ref{thm:var_general}.
\begin{proof}[Final part of the proof of Theorem \ref{thm:var_general}]
    \begin{enumerate}[label=\textbf{\textsc{Step \arabic*:}}, fullwidth]
        \setcounter{enumi}{3}
        \item We show the positivity of non-negative extremal functions.
            From the continuity of $u$, the set $A = \{x\colon u(x)=0\}$ is closed.
            By contradiction, suppose that $A$ is non-empty. 
            By Lemma \ref{lemma:strict_positive}, it follows that $A$ is also open.
            Since $X$ is connected, it follows that $A = X$ and therefore $u\equiv 0$.
            This however contradicts the fact that $\|u\|_2 =1$.
    \end{enumerate}
\end{proof}
Finally we address the proof of Corollary \ref{corollary:var_log_transform}.
\begin{proof}[Proof of Corollary \ref{corollary:var_log_transform}]
    Let $\phi(t)=\log t,t>0$ and $u$ be an arbitrary non-negative extremal function of the variational problem \eqref{eq:var_pbm_general}. 
    By Theorem \ref{thm:var_general}, we know that $0<c\leq u \leq C$ for some some positive constants $c$ and $C$ and that $u$ is Lipschitz.
    Since $\phi$ is a $C^2$-function with bounded first and second derivatives on $[c,C]$, it follows by the chain rule that $v:=\phi(u)$ is Lipschitz, $v$ is in $D(\Delta)$ and
    \begin{equation}
        \Delta(\phi(u))
        =
        \phi'(u)\Delta u + \phi''(u)|\nabla u|^2
        =
        \frac{1}{u}\Delta u - \frac{1}{u^2}|\nabla u|^2.
    \end{equation}
    By \eqref{eq:var_pbm_pde} we get
    \begin{equation}
        -\Delta v = |\nabla v|^2 + \alpha_2 v +(\lambda-\alpha_1).
    \end{equation}
    Taking $\tilde{v}=v + (\lambda-\alpha_1)/\alpha_2$.
    By locality of the minimal weak upper gradient and the Laplacian, we obtain that $\tilde{v}\in D(\Delta)$ satisfies $-\Delta \tilde{v}=|\nabla \tilde{v}|^2 + \alpha_{2}\tilde{v}$.
    If $\lambda\neq\alpha_1$, then Theorem \ref{thm:var_general} implies that $u$ is non-constant, which also implies that $\tilde{v}$ is non-constant.
    Since $\tilde{v}$ satisfies \eqref{eq:var_log_transform_pde}, we obtain the result.
\end{proof}

\section{Li-Yau type inequality for logarithmic extremal functions}
In this section, we derive a Li-Yau type estimate for the Lipschitz solutions of the equation $-\Delta v=|\nabla v|^2 + \alpha v$, whose existence is guaranteed by Corollary \ref{corollary:var_log_transform}.
In particular, based on the regularity and positivity results obtained in the previous section, this Li-Yau estimate holds for any logarithmic transform of non-negative extremal functions of \eqref{eq:var_pbm_general}.

\begin{theorem}\label{thm:estimate_Li_Yau}
    Let $(X,d,m)$ be a compact $\mathrm{RCD}^*(K,N)$ space with $K\in \mathbb{R}$ and $N$ in $(2,\infty)$.
    Let $v\in \mathrm{Lip}\cap D(\Delta)$ such that $-\Delta v=|\nabla v|^2 + \alpha v$ for some $\alpha>0$. 
    Then, for all $0< \beta <1$ it holds
    \begin{equation}\label{eq:Li_Yau_result_1}
        |\nabla v|^2 + (\alpha -\beta K)v
        \leq
        \frac{N\alpha(1-\beta)}{4\beta}\left(1-\frac{\beta((2-\beta)K-\alpha)}{2\alpha(1-\beta)}\right)^2,\quad \text{$m$-a.e.}
    \end{equation}
\end{theorem}
\begin{corollary}\label{coro:Li_Yau_extremal}
    Let $u$ be any non-negative extremal function of \eqref{eq:var_pbm_general} with the log-Sobolev constant $\lambda(\alpha_1,\alpha_2)$ and $\alpha_2>0$.
    Then $v:=\log u + (\lambda-\alpha_1)/\alpha_2$ satisfies the Li-Yau type estimate \eqref{eq:Li_Yau_result_1} with $\alpha=\alpha_2$.
    Moreover, if $K>0$ and $0<\alpha_2\leq K$, then any non-negative extremal function is constant.
\end{corollary}
The proof of Theorem \ref{thm:estimate_Li_Yau} is divided into three parts: In the first step, we show the regularity for $|\nabla v|^2+(\alpha-\beta Kv)$ using the weak Bochner inequality \eqref{eq:weak_Bochner}.
In the second step, following the similar computation arguments as in \cite{wang1999harnack} we derive a lower bound of the absolutely continuous part of $\bm{\Delta}(|\nabla v|^2 +(\alpha-\beta Kv))$, where all inequalities are understood in the $m$-a.e. sense.
In the last step, we make use of a slightly generalized Omori-Yau type maximum principle proved in Appendix \ref{appendix} together with a ``good'' cut-off function inspired by \cite{mondino2019structure} to derive the desired Li\--Yau estimate.

\begin{proof}[Proof of Theorem \ref{thm:estimate_Li_Yau}]
    Recall that $\mathrm{Test}^{\infty}:=\{f\in \mathrm{Lip}\cap D(\Delta)\cap L^{\infty}: \Delta f\in W^{1,2}\cap L^{\infty}\}$ and $\mathrm{Test}^{\infty}_{+}:=\{f\in \mathrm{Test}^{\infty}: f\geq 0 \text{ $m$-a.e. on $X$}\}$. 
    \begin{enumerate}[label=\textbf{\textsc{Step \arabic*:}}, fullwidth]
        \item We claim that $|\nabla v|^2$ is in $W^{1,2}\cap L^{\infty}$ and $|\nabla v|^2$ is in $D(\bm{\Delta})$ with $\bm{\Delta}^{s}(|\nabla v|^2)\geq 0$.
            Indeed, by the assumption that $v$ is in $D(\Delta)$, it follows that $v$ belongs to $ H^{2,2}$ and
            \begin{equation}
                |\nabla |\nabla v|^2|
                \leq
                2 |\mathrm{Hess} v|_{HS}|\nabla v|,
            \end{equation}
            with $|\mathrm{Hess}v|_{HS}\in L^2$.
            By the fact that $|\nabla v|$ is in $L^{\infty}$, we obtain that $|\nabla v|^2$ is also in $W^{1,2}$.
            We now show that $|\nabla v|^2\in D(\bm{\Delta})$.
            For any $\phi\in \mathrm{Test}^{\infty}_{+}$, by the weak Bochner inequality \eqref{eq:weak_Bochner}, it follows that
            \begin{equation}\label{eq:est_2}
                \int_X |\nabla v|^2 \Delta \phi dm
                \geq
                2\int_X \phi \left( \frac{(\Delta v)^2}{N} + \langle \nabla v, \nabla \Delta v \rangle + K |\nabla v|^2\right) d m
                =
                \int_X \phi d\mu,
            \end{equation}
            where $\mu=2((\Delta v)^2/N +\langle \nabla v, \nabla \Delta v\rangle + K|\nabla v|^2)m$.
            By the standard regularization via the mollified heat flow, see \citep[Corollary 6.2.17]{gigli2020lectures}, the inequality \eqref{eq:est_2} holds for all $\phi\in \mathrm{Lip}_{bs}^+$.
            Then by \citep[Proposition 6.2.16]{gigli2020lectures}, it follows that $|\nabla v|^2\in D(\bm{\Delta})$ and that
            \begin{equation}\label{eq:est_3}
                \bm{\Delta}\left(|\nabla v|^2\right)
                \geq
                2\left( \frac{(\Delta v)^2}{N} + \langle \nabla v, \nabla \Delta v\rangle + K|\nabla v|^2 \right)\cdot m.
            \end{equation}
            In particular we obtain that $\bm{\Delta}^{s}(|\nabla v|^2)\geq 0$.
        \item We provide a lower bound for $(\bm{\Delta}^{ac}g)$ based on the inequality \eqref{eq:est_3} where $g:=|\nabla v|^2 +(\alpha - \beta K)v$ for $0<\beta<1$.
            First note that since $-\Delta v=|\nabla v|^2 +\alpha v$, it follows that $-\Delta v = g + \beta K v$.
            Hence, from the first step, we get that $g$ is in $D(\bm{\Delta})$ and $\bm{\Delta}^{s}g=\bm{\Delta}^{s}(|\nabla v|^2)\geq 0$.
            Then inequality \eqref{eq:est_3} implies that
            \begin{multline}\label{eq:est_4}
                \bm{\Delta}^{ac}g
                =
                \bm{\Delta}^{ac}(|\nabla v|^2 +(\alpha - \beta K)v)\\
                \geq
                2\frac{(\Delta v)^2}{N} + 2\langle \nabla v, \nabla \Delta v\rangle + 2K|\nabla v|^2 + (\alpha - \beta K)\Delta v.
            \end{multline}
            Plugging $\Delta v=-g-\beta K v$ into \eqref{eq:est_4}, it follows that
            \begin{multline}\label{eq:Li_Yau_est_add_1}
                \bm{\Delta}^{ac}g
                \geq
                2\frac{(g+\beta K v)^2}{N} - 2\langle \nabla v, \nabla (g + \beta K v) \rangle + 2K|\nabla v|^2  - (\alpha - \beta K)(g+\beta K v)\\
                =
                \frac{2}{N}\left(g^2 + 2\beta K g v + \beta^2K^2 v^2\right) -2\beta K |\nabla v|^2 - 2\langle \nabla v, \nabla g\rangle + 2K|\nabla v|^2\\
                - (\alpha - \beta K)g - \beta K (\alpha-\beta K)v.
            \end{multline}
            Plugging identity $|\nabla v|^2= g-(\alpha-\beta K)v$ into the right hand of \eqref{eq:Li_Yau_est_add_1}, it follows that
            \begin{multline}\label{eq:Li_Yau_add_0}
                \bm{\Delta}^{ac}g
                \geq
                \frac{2}{N}g^2 + \left(\frac{4\beta K}{N}v + (2-\beta)K - \alpha\right)g
                + \frac{2\beta^2K^2}{N}v^2 \\
                -K(2-\beta)(\alpha-\beta K)v - 2\langle \nabla v, \nabla g \rangle\\
                =
                \frac{2}{N}\left[ g + \left(\beta K v + \frac{N[(2-\beta)K-\alpha]}{4}\right)\right]^2 - \frac{2}{N}\left( \beta K v + \frac{N[(2-\beta)K-\alpha]}{4}\right)^2\\
                + \frac{2\beta^2K^2}{N}v^2 -K(2-\beta)(\alpha-\beta K)v - 2\langle \nabla v, \nabla g \rangle.
            \end{multline}
            For $a=N((2-\beta)K-\alpha)$ and $b=N\alpha(1-\beta)/\beta$, inequality \eqref{eq:Li_Yau_add_0} simplifies to
            \begin{equation}\label{eq:est_Bochner_ineq_AC}
                \bm{\Delta}^{ac}g
                \geq
                \frac{2}{N}(g + \beta K v + \frac{a}{4})^2 - \frac{2}{N}(K\beta bv) - \frac{2}{N}\frac{a^2}{16} - 2\langle \nabla v, \nabla g\rangle.
            \end{equation}
        \item We show that the assumptions of Lemma \ref{lemma:maximum_principle} are valid for $g$ based on the estimate \eqref{eq:est_Bochner_ineq_AC}.
            Let $D:=\mathrm{diam}(X)$ and fix $0<R< D/4$.
            Since $X$ is compact and $g\in L^{\infty}$, we can find $x_0$ in $X$ such that:
            \begin{equation}
                M_1:=\esssup_{B(x_0,R)}g = \esssup_{X}g
            \end{equation}
            Define
            \begin{equation*}
                M_2 := \esssup_{X \setminus B(x_0, R)} g \leq M_1
            \end{equation*}
            By our choice of $R$ and the doubling property of $X$, we have $m(B(x_0,R))>0$ and $m(X\setminus B(x_0,R))>0$.
            Without loss of generality, we assume that $M_1>0$, otherwise nothing needs to be shown.
            We consider different cases for $M_1$ and $M_2$. 
            \begin{enumerate}[fullwidth]
                \item[Case 1:]$M_1>M_2$.
                    By the regularity result in Step 1, we can apply Lemma \ref{lemma:maximum_principle} to $g$.
                    Hence taking $w=2v$ which is in $W^{1,2}\cap \mathrm{Lip}_b$, it follows from Lemma \ref{lemma:maximum_principle} that we can find a sequence $(x_j)\subseteq X$ such that $g(x_j)>M_1 - 1/j$ and
                    \begin{equation}\label{eq:Li_Yau_add_2}
                        \bm{\Delta}^{ac}g(x_j)+\langle \nabla g, \nabla w \rangle(x_j)\leq 1/j.
                    \end{equation} 
                    Plugging \eqref{eq:est_Bochner_ineq_AC} into \eqref{eq:Li_Yau_add_2} and letting $h_{v,j}= (K\beta bv+a^2/16 + N/(2j))^{1/2}(x_j)$, it follows that
                    \begin{multline}
                        g(x_j)
                        \leq
                        -\beta K v(x_j) - \frac{a}{4} + h_{v,j}\\
                        =
                        -\frac{1}{b}h_{v,j}^2 + h_{v,j} + \frac{a^2}{16b} - \frac{a}{4} + \frac{N}{2bj}\\
                        =
                        -\frac{1}{b}\left(h_{v,j}-\frac{b}{2}\right)^2 + \frac{b}{4}\left(1-\frac{a}{2b}\right)^2 + \frac{N}{2bj}.
                    \end{multline}
                    For $j\rightarrow \infty$ together with $g(x_j)>M_1-1/j$ and $b>0$, we obtain that
                    \begin{equation}
                        \esssup_{X}g
                        \leq
                        \frac{b}{4}\left(1-\frac{a}{2b}\right)^2.
                    \end{equation}

                \item[Case 2:]$M_1=M_2$.
                    Let $\varepsilon\in (0,1/2)$.
                    By \citep[Lemma 3.1]{mondino2019structure}, we can find a Lipschitz cut-off function $\psi:X\rightarrow \mathbb{R}$ with $\psi\in D(\Delta)$ such that $0\leq \psi\leq 1$ and $\psi\equiv 1$ on $B(x_0,R)$ and $\mathrm{supp}(\psi)\subseteq B(x_0,2R)$, and 
                    \begin{equation}\label{eq:Naber_original_cut_off}
                        R^2|\Delta \psi| + R|\nabla \psi|
                        \leq C,
                    \end{equation}
                    where $C>0$ is a constant depending only on $K,N$ and $R$.
                    Let $\phi_{\varepsilon}:X\rightarrow \mathbb{R}$ be defined as $\phi_{\varepsilon}:=1-\varepsilon + \varepsilon \psi$ and let $G_{\varepsilon}:=\phi_{\varepsilon}\cdot g$.
                    Note that since $\phi_{\varepsilon}\in D(\Delta)$, it holds that $\bm{\Delta}\phi_{\varepsilon}=\Delta \phi_{\varepsilon}\cdot m$.
                    Together with the fact that $\phi_{\varepsilon}$ is continuous, by the Leibniz rule for the measure-valued Laplacian, it follows that
                    \begin{equation}
                        \bm{\Delta}(G_{\varepsilon})
                        =
                        \phi_{\varepsilon} \bm{\Delta}g + g\Delta\phi_{\varepsilon}\cdot m + 2\langle \nabla\phi_{\varepsilon}, \nabla g \rangle\cdot m. 
                    \end{equation}
                    
                    Together with the result in the first step of this proof that $\bm{\Delta}^{s}g\geq 0$, we deduce that
                    \begin{equation}
                        \bm{\Delta}^{s}(G_{\varepsilon})
                        =
                        \phi_{\varepsilon}\bm{\Delta}^{s}g
                        \geq 0.
                    \end{equation}
                    Furthermore, from \eqref{eq:Naber_original_cut_off}, it follows that
                    \begin{equation}\label{eq:Naber_cut_off}
                        \frac{|\nabla \phi_{\varepsilon}|^2}{\phi_{\varepsilon}}
                        \leq
                        \varepsilon^2\frac{C}{R^2}
                        \quad \text{and} \quad
                        |\Delta \phi_{\varepsilon}|
                        \leq \varepsilon \frac{C}{R^2}.
                    \end{equation}
                    Denote by $H_v$ the right side of \eqref{eq:est_Bochner_ineq_AC} without the last term $-2\langle \nabla v, \nabla g\rangle$.
                    Then, by using $g=G_{\varepsilon}/\phi_{\varepsilon}$ and the inequality \eqref{eq:est_Bochner_ineq_AC}, it follows that
                    \begin{multline}
                        \bm{\Delta}^{ac}(G_{\varepsilon})
                        =
                        \phi_{\varepsilon}\bm{\Delta}^{ac}g + \frac{G_{\varepsilon}}{\phi_{\varepsilon}}\Delta \phi_{\varepsilon} + 2\left\langle \nabla\phi_{\varepsilon}, \nabla\left(\frac{G_{\varepsilon}}{\phi_{\varepsilon}}\right)\right\rangle\\
                        \geq
                        \phi_{\varepsilon}\left(H_v - 2\left\langle \nabla v, \nabla g\right\rangle\right)
                        +2\left\langle \nabla \phi_{\varepsilon}, \nabla G_{\varepsilon}\right\rangle/\phi_{\varepsilon} 
                        +\frac{G_{\varepsilon}}{\phi_{\varepsilon}}\left( \Delta \phi_{\varepsilon}-2\frac{|\nabla \phi_{\varepsilon}|^2}{\phi_{\varepsilon}}\right).
                    \end{multline}
                    Using the estimate \eqref{eq:Naber_cut_off} with $\varepsilon^2<\varepsilon$, it follows that
                    \begin{multline}\label{eq:Li_Yau_add_3}
                        \bm{\Delta}^{ac}(G_{\varepsilon})
                        \geq
                        \phi_{\varepsilon}H_v -2\phi_{\varepsilon}\left( \left\langle \nabla v, \frac{\nabla G_{\varepsilon}}{\phi_{\varepsilon}}\right\rangle - \left\langle \nabla v, \frac{G_{\varepsilon}\nabla \phi_{\varepsilon}}{\phi_{\varepsilon}^2}\right\rangle\right)\\
                        + 2\left\langle \nabla \phi_{\varepsilon}, \nabla G_{\varepsilon}\right\rangle/\phi_{\varepsilon} - \varepsilon\|g\|_{\infty}\frac{C}{R^2}\\
                        \geq
                        \phi_{\varepsilon}H_v - 2\left\langle \nabla (v - \log \phi_{\varepsilon}), \nabla G_{\varepsilon} \right\rangle + 2\frac{G_{\varepsilon}}{\phi_{\varepsilon}}\langle \nabla v, \nabla \phi_{\varepsilon}\rangle - \varepsilon\|g\|_{\infty}\frac{C}{R^2}\\
                        \geq
                        \phi_{\varepsilon}H_v - 2\langle \nabla(v-\log \phi_{\varepsilon}), \nabla G_{\varepsilon}\rangle - \varepsilon\|g\|_{\infty}\| |\nabla v|\|_{\infty}\frac{C}{R} - \varepsilon\|g\|_{\infty}\frac{C}{R^2}.
                    \end{multline}
                    Since $\esssup_{X\setminus B(x_0,2R)}g \leq \esssup_{X\setminus B(x_0,R)}g<M_1$, by the definition of $\phi_{\varepsilon}$, it follows that
                    \begin{equation}
                        \esssup_{X\setminus B(x_0,2R)}G_{\varepsilon}
                        \leq
                        (1-\varepsilon)M_1
                        <
                        M_1
                        =
                        \esssup_{B(x_0,R)}G_{\varepsilon}
                        =
                        \esssup_{B(x_0,2R)}G_{\varepsilon}.
                    \end{equation}
                    The doubling property and $R<D/4$ implies that $m(B(x_0,2R))>0$ as well as $m(X\setminus B(x_0,2R))>0$, which together with the fact that $\bm{\Delta}^{s}(G_{\varepsilon})\geq 0$ and $G_{\varepsilon}\in W^{1,2}\cap \mathrm{Lip}_b$, allows us to apply Lemma \ref{lemma:maximum_principle} to $G_{\varepsilon}$.
                    Taking $w_{\varepsilon}=2v - 2\log \phi_{\varepsilon}\in W^{1,2}\cap \mathrm{Lip}_b$, by Lemma \ref{lemma:maximum_principle}, it follows that there exists a sequence $(x_{j})\subseteq X$ such that $G_{\varepsilon}(x_{j})>\esssup_{X}G_{\varepsilon}-1/j$ and
                    \begin{equation}
                        \bm{\Delta}^{ac}G_{\varepsilon}(x_{j}) + \langle \nabla G_{\varepsilon}, \nabla w_{\varepsilon}\rangle(x_{j})
                        \leq 1/j.
                    \end{equation}
                    Plugging \eqref{eq:Li_Yau_add_3}, we obtain that
                    \begin{equation}
                        \phi_{\varepsilon}(x_j)H_v(x_j)
                        \leq
                        \frac{1}{j} + \varepsilon C_1,
                    \end{equation}
                    where $C_1=C(K,N,R)\|g\|_{\infty}\left( \||\nabla f|\|_{\infty}/R + 1/R^2\right)$.
                    Following the similar argument as in Case 1 and noting that $b>0$, we obtain that
                    \begin{equation}\label{eq:Li_Yau_add_4}
                        g(x_j)
                        \leq
                        \frac{b}{4}\left(1-\frac{a}{2b}\right)^2 + \frac{N(j^{-1}+\varepsilon C_1)}{2b \phi_{\varepsilon}(x_j)}.
                    \end{equation}
                    Since $\esssup_{X}G_{\varepsilon}=\esssup_{X}g=M_1$, so multiplying $\phi_{\varepsilon}(x_j)$ on both side of \eqref{eq:Li_Yau_add_4} and letting $j\rightarrow \infty$, it follows that
                    \begin{equation}\label{eq:est_6}
                        \esssup_{X}g
                        \leq
                        \frac{b}{4}\left(1-\frac{a}{2b}\right)^2 + \frac{\varepsilon NC_1}{2b}.
                    \end{equation}
                    The inequality \eqref{eq:est_6} holding for any $0<\varepsilon <1/2$, we send $\varepsilon$ to $0$ to obtain
                    \begin{equation}
                        \esssup_{X}g
                        \leq
                        \frac{b}{4}\left(1-\frac{a}{2b}\right)^2.
                    \end{equation}
            \end{enumerate}
            With Case 1 and Case 2, together with the definitions of $a$ and $b$, we obtain the inequality \eqref{eq:Li_Yau_result_1}.
    \end{enumerate}
\end{proof}

\begin{proof}[Proof of Corollary \ref{coro:Li_Yau_extremal}]
    By Theorem \ref{thm:var_general} and Corollary \ref{corollary:var_log_transform}, we know that $w:=\log u$ is in $\mathrm{Lip}\cap D(\Delta)$ and satisfies $-\Delta w=|\nabla w|^2 +\alpha_2 w + \lambda -\alpha_1$.
    Then by the definition of the Laplacian and locality of the minimal weak upper gradient, it follows that $v=w+(\lambda-\alpha_1)/\alpha_2$ is in $\mathrm{Lip}\cap D(\Delta)$ and satisfies the equation $-\Delta v=|\nabla v|^2 + \alpha_2 v$.
    Hence, by Theorem \ref{thm:estimate_Li_Yau}, the Li\--Yau estimate \eqref{eq:Li_Yau_result_1} holds for $v$.
    For the second claim, let $C(\beta)$ denote the term in the right-hand side of \eqref{eq:Li_Yau_result_1}.
    If $\alpha_2\in (0,K]$, then taking $\beta\in (0,\alpha_2/K)$.
    One can check that $C(\beta)$ goes to $0$ as $\beta\nearrow \alpha_2/K$.
    Hence, taking $\beta\nearrow \alpha_2/K$ on the both sides of \eqref{eq:Li_Yau_result_1}, it follows that
    \begin{equation*}
        \left|\nabla v\right|^2 = 0,\quad \text{$m$-a.e.}
    \end{equation*}
    Then the Sobolev-to-Lip property implies that $u$ is constant.
\end{proof}

\section{Applications}
In this section, we present applications of the regularity and positivity results in Theorem \ref{thm:var_general} and the Li-Yau type estimate in Theorem \ref{thm:estimate_Li_Yau}, which generalize the results in \citep{wang1999harnack} on the smooth Riemannian manifold to the $\mathrm{RCD}^*(K,N)$ spaces.

For the notational simplicity, we define the non-negative constants $C_1$ and $C_2$ as follows:
\begin{equation}
    C_1(\beta):=\alpha - \beta K
    \quad\text{and}\quad
    C_2(\beta):= \frac{N\alpha(1-\beta)}{4\beta}\left(1-\frac{\beta((2-\beta)K-\alpha)}{2\alpha(1-\beta)}\right)^2,
\end{equation}
where $0<\beta<1$.
Note that $C_1$ is positive whenever $\alpha>\max\{K,0\}$.

The first direct consequence is a Harnack type inequality for the non-negative extremal functions.
\begin{corollary}\label{coro:app_1}
    Let $(X,d,m)$ be a compact $\mathrm{RCD}^*(K,N)$ space with $K\in \mathbb{R}$ and $N$ in $(2,\infty)$.
    Suppose that $v\in \mathrm{Lip}$ satisfies $-\Delta v=|\nabla v|^2 +\alpha v$ for some $\alpha>\max\{K,0\}$.
    Then for any $x,y\in X$, it holds that
    \begin{equation}
        e^{v(x)}
        \leq
        e^{(1-\varepsilon)v(y)}\exp\left(\frac{C_1(\beta) d^2(x,y)}{4\varepsilon} + \frac{\varepsilon C_2(\beta)}{C_1(\beta)}\right),
    \end{equation}
    for any $0<\varepsilon<1$ and $\beta\in (0,1)$.
\end{corollary}
\begin{proof}
    First note that by Theorem \ref{thm:estimate_Li_Yau}, we have
    \begin{equation}\label{eq:application_1_1}
        |\nabla v |
        \leq
        \sqrt{C_2(\beta)-C_1(\beta)v},\quad \text{$m$-a.e.}
    \end{equation}
    Let $x_0,y_0\in X$ be arbitrary points in $X$ and let $\mu_0:=\frac{1}{m(B(x_0,r))}m|_{B(x_0,r)}$ and $\mu_1:=\frac{1}{m(B(y_0,r))}m|_{B(y_0,r)}$ for $0<r<\mathrm{diam}(X)/3$.
    By \citep[Corollary 1.2]{rajala2014non_branching}, $\mathrm{RCD}^*(K,N)$ spaces are essentially non-branching.
    Together with \cite[Corollary 5.3]{cavalletti2017optimal}, there exists a unique $W_2$-geodesic $(\mu_t)_{\{t\in [0,1]\}}$ joining $\mu_0$ and $\mu_1$ with $\mu_t\leq Cm$ for any $t$ in $[0,1]$ for some $C>0$, and a test plan $\pi$ in $\mathcal{P}(C([0,1];X))$ such that $(e_t)_{\sharp}\pi=\mu_t$ for any $t$ in $[0,1]$ and $\pi$ is concentrated on the set $\mathrm{Geo}(X)$ of geodesics on $X$.
    By the Fubini's theorem together with the fact that $\pi$ has bounded compression and the inequality \eqref{eq:application_1_1}, it follows that for $\pi$-a.s $\gamma\in \mathrm{Geo}(X)$,
    \begin{equation*}
        \left|\nabla v(\gamma_t)\right|\leq \sqrt{C_2(\beta)-C_1(\beta)v(\gamma_t)},\quad \text{for a.e. $t\in [0,1]$}.
    \end{equation*}
    Let $\gamma\in \mathrm{Geo}(X)$ be any such a geodesic.
    By the continuity of $v$, let $s_1$ and $s_2$ in $[0,1]$ be the maximum and minimum points respectively, that is
    \begin{equation}
        v(\gamma_{s_1})=\max_{s\in[0,1]}v(\gamma_s)
        \quad\text{and}\quad
        v(\gamma_{s_2})=\min_{s\in [0,1]}v(\gamma_s).
    \end{equation}
    Then by the definition of weak upper gradients (see \cite[Proposition 1.20]{gigli2020lectures}), it follows that
    \begin{multline}
        \left| v(\gamma_{s_1}) - v(\gamma_{s_2}) \right|
        \leq
        \int_{\min(s_1,s_2)}^{\max(s_1,s_2)}\left|\nabla v \right|(\gamma_t) \left|\dot{\gamma}_t \right|dt
        \leq
        d(\gamma_0,\gamma_1)\sqrt{C_2(\beta)-C_1(\beta)v(\gamma_{s_2})}.
    \end{multline}
    For $0<\varepsilon<1$, together with $v(\gamma_0)\leq v(\gamma_{s_1})$, it follows that
    \begin{equation}
        v(\gamma_0)
        \leq
        (1-\varepsilon)v(\gamma_{s_2}) + \varepsilon v(\gamma_{s_2}) + d(\gamma_0,\gamma_1)\sqrt{C_2(\beta)-C_1(\beta)v(\gamma_{s_2})}.
    \end{equation}
    Let $a=\sqrt{C_2(\beta)-C_1(\beta)v(\gamma_{s_2})}$. 
    Since $C_1(\beta)>0$, it follows that
    \begin{multline}\label{eq:app_coro1_1}
        v(\gamma_0)
        \leq
        (1-\varepsilon)v(\gamma_{s_2}) - \frac{\varepsilon}{C_1(\beta)}a^2 + d(\gamma_0,\gamma_1)a + \frac{\varepsilon C_2(\beta)}{C_1(\beta)}\\
        \leq
        (1-\varepsilon)v(\gamma_{1}) + \frac{C_1(\beta) d^2(\gamma_0,\gamma_1)}{4\varepsilon} + \frac{\varepsilon C_2(\beta)}{C_1(\beta)}.
    \end{multline}
    Integrating \eqref{eq:app_coro1_1} with respect to $\pi$ on the both sides and noting that $(e_t)_{\sharp}\pi=\mu_t$, it follows that
    \begin{equation}
        \int v(x)d\mu_0
        \leq
        \int \left( (1-\varepsilon)v(y) + \frac{C_1(\beta) d^2(x,y)}{4\varepsilon} + \frac{\varepsilon C_2(\beta)}{C_1(\beta)} \right) d\tilde{\pi}(x,y),
    \end{equation}
    where $\tilde{\pi}=(e_0,e_1)_{\sharp}\pi$ is the optimal transport plan between $\mu_0$ and $\mu_1$.
    Since $\mu_0$ and $\mu_1$ weakly converges to $\delta_{x_0}$ and $\delta_{y_0}$ in the duality of $C_b$ as $r\rightarrow 0$, it follows that, up to a subsequence, $\tilde{\pi}$ weakly converges to $\delta_{x_0}\times \delta_{y_0}$.
    Using the identity $u=\exp(v)$, we obtain that
    \begin{equation}
        u(x_0)
        \leq
        u(y_0)^{1-\varepsilon}\exp\left(\frac{C_1(\beta) d^2(x_0,y_0)}{4\varepsilon} + \frac{\varepsilon C_2(\beta)}{C_1(\beta)}\right).
    \end{equation}
\end{proof}
\begin{remark}
    Note that using the result proved in \cite{cheeger1999} that the minimal weak upper gradient of Lipschitz function coincides with its local Lipschitz slope in the complete doubling metric space supporting the weak $(1,p)$-Poincaré inequality for $p>1$, Corollary \ref{coro:app_1} can be shown directly following the lines of \citep[Corollary 2.2]{wang1999harnack} instead of using the optimal transport method.
\end{remark}

The next corollary addresses the estimates on the upper and lower bounds of $v$ based on the dimension-free Harnack inequality on the $\mathrm{RCD}(K,\infty)$ space in \citep{huaiqian2016dimension}.
The proof of which is essentially the same one as in the Riemannian case from \citep[Corollary 2.4]{wang1999harnack}.
For the sake of completeness, we provide the proof.
\begin{corollary}\label{coro:app_2}
    Let $(X,d,m)$ be a compact $\mathrm{RCD}^*(K,N)$ space with $K\in \mathbb{R}$ and $N$ in $(2,\infty)$.
    Then for any non-negative extremal function $u$ of \eqref{eq:var_pbm_general} with the log Sobolev constant $\lambda(\alpha_1,\alpha_2)$ with $\alpha_2>\max\{K,0\}$, it holds that
    \begin{align}
        &\sup_{X}\log u \leq \frac{\lambda-\alpha_1}{\alpha_2} + C_3(\alpha_2,K) \mathrm{diam}(X)^2,\\
        &\inf_{X}\log u
        \geq
        - \frac{\lambda-\alpha_1}{\alpha_2} - 2C_3(\alpha_2, K) \mathrm{diam}(X)^2
    \end{align}
    where $C_3(\alpha_2,K)=\alpha_2$ when $K\geq 0$ and $C_3(\alpha_2,K)=K/(\exp(K/\alpha_2)-1)$ when $K<0$.
    In particular, it holds that $|\nabla \log u|^2 \leq C_2(\beta)+2 C_1(\beta)C_3(\alpha_2,K) \mathrm{diam}(X)^2$ $m$-a.e. for any $0<\beta<1$.
\end{corollary}
\begin{proof}
    Let $D:=\mathrm{diam}(X)$ and $x$ and $y$ be the maximum and minimum point of $u$ on $X$, respectively, the existence of which is guaranteed by the regularity of non-negative extremal functions proved in Theorem \ref{thm:var_general}.
    
    As for the upper bound, firstly, by the dimension-free Harnack inequality on $\mathrm{RCD}(K,\infty)$ spaces shown in \citep[Theorem 3.1]{huaiqian2016dimension} for $p>1$, it follows that
    \begin{equation}\label{eq:coro_app2_1}
        \left(P_t u(x)\right)^p
        \leq
        P_t(u^p)(z) \exp\left(\frac{pKd^2(x,z)}{2(p-1)(e^{2Kt}-1)}\right),\quad \text{for any } z\in X.
    \end{equation}
    Since $\|P_t u^2\|_1=\|u^2\|_1=1$ by the mass-preserving property of the heat semigroup, we deduce from \eqref{eq:coro_app2_1} by taking $p=2$ that
    \begin{equation}\label{eq:coro_app2_2}
        \left(P_{t}u(x)\right)^2
        \leq
        \exp\left(\frac{KD^2}{e^{2Kt}-1}\right).
    \end{equation}
    Secondly, by the equation \eqref{eq:var_pbm_pde} and the commutation between $P_t$ and $\Delta$, it follows that
    \begin{multline}\label{eq:coro_app2_revision_add_1}
        P_tu(z) - P_su(z)
        =
        \int_{s}^{t}P_{\tau}\left(\Delta u\right)(z)d\tau\\
        =
        -\int_{s}^{t}P_\tau\left(\alpha_2 u\log u +(\lambda-\alpha_1)u\right)(z)d\tau,\quad \text{for any $0<s<t$},
    \end{multline}
    for $m$-a.e $z\in X$.
    By the regularization of the heat semigroup, the equality \eqref{eq:coro_app2_revision_add_1} holds for any $z\in X$.
    Since $u$ is positive and $x$ is the maximum point, it follows that $(u\log u)(z)\leq u(z)\log u(x)$ for any $z$ in $X$ which by the comparison principle of the heat semigroup, yields
    \begin{equation}
        P_tu(x) - P_su(x)
        \geq
        -\alpha_2\log u(x)\int_{s}^{t}P_{\tau}u(x)d\tau
        - (\lambda-\alpha_1)\int_{s}^{t}P_{\tau}u(x)d\tau.
    \end{equation}
    Gr\"onwall's inequality further implies that
    \begin{equation}\label{eq:coro_app2_x}
        P_tu(x)
        \geq
        u(x)\exp\left(-\alpha_2 t \log u(x) -(\lambda-\alpha_1)t\right)
        =
        e^{-(\lambda-\alpha_1)t}u(x)^{1-\alpha_2 t}.
    \end{equation}
    So, together with \eqref{eq:coro_app2_2}, it follows that for any $0<t <1/\alpha_2$:
    \begin{equation}
        \log u(x)
        \leq
        \frac{(\lambda-\alpha_1)t}{1-\alpha_2t}+\frac{KD^2}{2(1-\alpha_2t)(e^{2Kt}-1)}.
    \end{equation}
    By using $K/(e^{2Kt}-1)\leq 1/(2t)$ when $K>0$ and letting $t=1/(2\alpha_2)$, it follows that
    \begin{equation}
        \log u(x) \leq \frac{\lambda-\alpha_1}{\alpha_2} + C_3(\alpha_2,K) D^2,
    \end{equation}
    where $C_3(\alpha_2,K)=\alpha_2$ when $K\geq 0$ and $C_3(\alpha_2,K)=K/(\exp(K/\alpha_2)-1)$ when $K<0$.
    
    As for the lower bound, from $(u\log u)(z)\geq u(z)\log u(y)$ and similar arguments as above, it follows that
    \begin{equation}\label{eq:coro_app2_y}
        P_tu(y)
        \leq
        e^{-(\lambda-\alpha_1)t}u(y)^{1-\alpha_2 t}.
    \end{equation}
    Taking $z=y$ in the dimension-free inequality \eqref{eq:coro_app2_1} and plugging \eqref{eq:coro_app2_x} and \eqref{eq:coro_app2_y}, it follows that
    \begin{multline}
        e^{-p(\lambda-\alpha_1)t}u(x)^{p-p\alpha_2 t}
        \leq
        \left(P_tu(x)\right)^p
        \leq
        P_t(u^p)(y)\exp\left(\frac{pKD^2}{2(p-1)(e^{2Kt}-1)}\right)\\
        \leq
        u(x)^{p-1}(P_tu)(y)\exp\left(\frac{pKD^2}{2(p-1)(e^{2Kt}-1)}\right)\\
        \leq
        u(x)^{p-1}e^{-(\lambda-\alpha_1)t}u(y)^{1-\alpha_2t}\exp\left(\frac{pKD^2}{2(p-1)(e^{2Kt}-1)}\right),
    \end{multline}
    where in the third inequality we use the comparison principle of the heat semigroup that $P_tu\leq \sup u=u(x)$, and in the last inequality we use the inequality \eqref{eq:coro_app2_y}.
    After reorganizing the inequality, it follows that
    \begin{equation}
        u(x)^{1-p\alpha_2 t}
        \leq
        e^{(p-1)(\lambda-\alpha_1)t}u(y)^{1-\alpha_2t}\exp\left(\frac{pKD^2}{2(p-1)(e^{2Kt}-1)}\right).
    \end{equation}
    Taking $p=1/(\alpha_2t)$ and $t=1/(2\alpha_2)$, we obtain that
    \begin{equation}
        \log u(y)
        \geq
        - \frac{\lambda-\alpha_1}{\alpha_2} - 2 C_3(\alpha_2,K) D^2,
    \end{equation}
    where we use the inequality $K/(e^{2Kt}-1)\leq 1/(2t)$ when $K>0$ again.
\end{proof}

As a final application, following the similar methods as in \cite[Lemma 5.3]{profeta2015sharp} together with Theorem \ref{thm:var_general}, we recover the classical result from \cite[Theorem 5.7.4]{bakry2014analysis} that any non-negative extremal function with the log-Sobolev constant $\lambda(\alpha_1,\alpha_2)$ is constant when $K>0$ and $0<\alpha_2\leq KN/(N-1)$.
\begin{corollary}
    Let $(X,d,m)$ be a $\mathrm{RCD}^*(K,N)$ space with $K>0$ and $N\in (2,\infty)$.
    Then any non-negative extremal function of \eqref{eq:var_pbm_general} with the log-Sobolev constant $\lambda(\alpha_1, \alpha_2)$ is constant whenever $0< \alpha_2 \leq KN/(N-1)$.
\end{corollary}
\begin{proof}
    Let $u$ be an arbitrary non-negative extremal function of \eqref{eq:var_pbm_general} with a log-Sobolev constant $\lambda(\alpha_1,\alpha_2)$ with $0<\alpha_2 \leq KN/(N-1)$, and  define $v=\log u$.
    Let further $a$, $b$, and $d$ be real numbers to be determined later.
    We first estimate $\int e^{bv}(\Delta v)^2dm$ from both the PDE equation \eqref{eq:var_pbm_pde} and the weak Bochner inequality \eqref{eq:weak_Bochner}, and then derive the desired result.
    \begin{enumerate}[label=\textbf{\textsc{Step \arabic*:}}, fullwidth]
        \item We first derive the estimate from the equation \eqref{eq:var_pbm_pde}.
            Let $\alpha_3:=\lambda-\alpha_1$ in \eqref{eq:var_pbm_pde}.
            By the Lipschitz regularity of $v$ and the regularity of $|\nabla v|^2$ proved in Step 1 in Theorem \ref{thm:estimate_Li_Yau}, it follows that $\phi:=e^{(b-1)v}\Delta v\in W^{1,2}$ and $\psi:=e^{(b-1)v}|\nabla v|^2\in W^{1,2}$.
            On the one hand, applying $\phi$ and then $\psi$ to the right-hand side of \eqref{eq:var_pbm_pde} it follows that
            \begin{multline}\label{eq:app3_1}
                I:= \int (\alpha_2 v + \alpha_3)e^{v}\phi dm = \int (\alpha_2 v + \alpha_3)e^{v}e^{(b-1)v}\Delta v dm\\
                =
                -\alpha_2 \int |\nabla v|^2 e^{bv}dm - b\int (\alpha_2 v+\alpha_3)|\nabla v|^2 e^{bv}dm\\
                =-\alpha_2 \int |\nabla v|^2 e^{bv}dm - b\int (\alpha_2 v+\alpha_3)e^v\psi dm\\
                =-\alpha_2 \int |\nabla v|^2 e^{bv}dm + b\int \Delta(e^v)e^{(b-1)v}|\nabla v|^2 dm\\
                =-\alpha_2 \int |\nabla v|^2 e^{bv}dm + b\int |\nabla v|^4 e^{bv}dm +b \int (\Delta v)|\nabla v|^2 e^{bv}dm.
            \end{multline} 
           On the other hand, applying $\phi$ to the left-hand side of \eqref{eq:var_pbm_pde}, we get
            \begin{equation}
                I=
                \int -\Delta(e^v) e^{(b-1)v}\Delta v dm
                =
                -\int (\Delta v)|\nabla v|^2 e^{bv}dm - \int (\Delta v)^2 e^{bv}dm.
            \end{equation}
            Showing that
            \begin{equation}\label{eq:app3_ineq_1}
                \int (\Delta v)^2 e^{bv}dm
                =
                \alpha_2 \int |\nabla v|^2 e^{bv}dm - b\int |\nabla v|^4 e^{bv}dm -(b+1)\int (\Delta v)|\nabla v|^2 e^{bv}dm.
            \end{equation}
        \item We derive the estimate from the weak Bochner inequality \eqref{eq:weak_Bochner}.
            Note that $f:=e^{av}$ and $g:=e^{dv}$ satisfies the regularity requirement in \eqref{eq:weak_Bochner}.
            So plugging $f$ and $g$ into \eqref{eq:weak_Bochner}, it follows that the left-hand side of \eqref{eq:weak_Bochner} can be expressed as
            \begin{equation}\label{eq:app3_3}
                \frac{1}{2}\int \Delta(g)|\nabla f|^2dm
                =
                \frac{a^2d}{2}\int e^{(2a+d)v}(\Delta v)|\nabla v|^2 dm + \frac{a^2d^2}{2}\int e^{(2a+d)v}|\nabla v|^4dm,
            \end{equation}
            and the right-hand side of \eqref{eq:weak_Bochner} can be expressed as
            \begin{multline}\label{eq:app3_4}
                -\frac{N-1}{N}\int g(\Delta f)^2 dm -\int \Delta f \langle \nabla g, \nabla f \rangle dm + K\int g|\nabla f|^2dm\\
                =
                -a^2\frac{N-1}{N}\int e^{(2a+d)}(\Delta v)^2 dm
                - a^2\left(2a\frac{N-1}{N}+d\right)\int e^{(2a+d)v}(\Delta v)|\nabla v|^2 dm \\
                - a^2\left(a^2\frac{N-1}{N}+ad\right)\int e^{(2a+d)v}|\nabla v|^4 + a^2 K \int e^{(2a+d)v}|\nabla v|^2 dm.
            \end{multline}
            From \eqref{eq:app3_3} and \eqref{eq:app3_4}, it follows that \eqref{eq:weak_Bochner} reads as follows:
            \begin{multline}\label{eq:app3_ineq_2}
                \int e^{(2a+d)v}(\Delta v)^2 dm
                \geq
                \frac{KN}{N-1}\int e^{(2a+d)v}|\nabla v|^2\\
                -
                \left(a^2 + \frac{N}{N-1}ad + \frac{N}{2(N-1)}d^2\right)\int e^{(2a+d)v}|\nabla v|^4 dm\\
                - \left(2a + \frac{3N}{2(N-1)}d\right)\int e^{(2a+d)v}(\Delta v)|\nabla v|^2 dm.
            \end{multline}

        \item We conclude by comparing the coefficients of \eqref{eq:app3_ineq_1} and \eqref{eq:app3_ineq_2} and choosing particular values for $b,a,d$. 
            Let $\varepsilon>0$ be determined later and let $b,a,d\in \mathbb{R}$ satisfy the following system of equations:
            \begin{equation}\label{eq:app3_system}
                \left\{
                    \begin{aligned}
                    &b = 2a +d,\\
                    &b-\varepsilon = a^2 +\frac{N}{N-1}ad + \frac{N}{2(N-1)}d^2,\\
                    &b+1 = 2a + \frac{3N}{2(N-1)}d.
                    \end{aligned}
                \right.
            \end{equation}
            This system \eqref{eq:app3_system} admits real-valued solutions if and only if $0<\varepsilon\leq 4N/(N+2)^2$.
            Hence, choosing arbitrary $\varepsilon\in (0, 4N/(N+2)^2]$ and
            \begin{align*}
                d &= \frac{2(N-1)}{N+2},\\
                a &=\frac{2}{N+2} + \sqrt{\frac{4N}{(N+2)^2}-\varepsilon},\\
                b &= 2a+d.
            \end{align*}
            By comparing \eqref{eq:app3_ineq_1} and \eqref{eq:app3_ineq_2}, it follows that
            \begin{equation}
                \left(\alpha_2 - \frac{KN}{N-1}\right)\int |\nabla v|^2 e^{bv}dm
                \geq
                \varepsilon \int |\nabla v|^4 e^{bv}dm.
            \end{equation}
            Since $0<\alpha_2\leq KN/(N-1)$, then $|\nabla v|$ has to be $0$, implying that $u$ is constant.
    \end{enumerate}
\end{proof}

\begin{appendix}
\section{Omori-Yau Maximum Principle}\label{appendix}
In the appendix, we provide a slightly generalized version of the Omori-Yau type maximum principle for the whole metric space of proper $\mathrm{RCD}(K,\infty)$ spaces with $K$ in $\mathbb{R}$ which may not support the doubling property. 
To show it, we first show the Kato's inequality in the proper $\mathrm{RCD}(K,\infty)$ setting whose proof follows a similar argument as in \cite{zhang2016local}.
For the sake of completeness, we provide the complete argumentation.

Beforehand, recall the definition of the weak Laplacian.
We say that an operator $L$ on $W^{1,2}_{loc}$ is the weak Laplacian provided that for each $f\in W^{1,2}_{loc}$, $Lf$ is a linear functional acting on $W^{1,2}\cap L^{\infty}$ with bounded support defined as:
\begin{equation}
    Lf(g):=
    -\int \langle \nabla f, \nabla g \rangle dm,\quad \text{for all } g\in W^{1,2}\cap L^{\infty}\text{ with bounded support}.
\end{equation}
For each $h$ in $W^{1,2}_{loc}\cap L^{\infty}$, $h\cdot Lf$ is the linear functional given by $h\cdot Lf(g):=Lf(hg)$ for each $g$ in $W^{1,2}\cap L^{\infty}$ with bounded support.
We say that $Lf$ is a signed Radon measure provided that there exists a signed Radon measure $\mu$ such that $Lf(g)=\int g d\mu$ for all $g\in W^{1,2}\cap L^{\infty}$ with bounded support. 
It is clear that in this case, we have $f\in D(\bm{\Delta})$ and $Lf=\bm{\Delta}f$.
For $w\in W^{1,2}\cap L^{\infty}$ and $m_w:=e^{w}\cdot m$, $L_w$ denotes the weighted weak Laplacian on $W^{1,2}_{loc}$ defined as
\begin{equation}
    L_wf(g):= - \int \langle \nabla f, \nabla g \rangle dm_w,\quad \text{for all }g\in W^{1,2}\cap L^{\infty}\text{ with bounded support}.
\end{equation}
It is also easy to check that $L_wf=e^w\cdot(Lf + \langle \nabla w, \nabla f \rangle m)$.
When $L_wf$ is a signed Radon measure, we denote by $L_wf=(L^{ac}_wf)\cdot m_w + L_w^{s}f$ its Lebesgue decomposition with respect to $m_w$.
Finally, we remark that using the similar arguments as in \cite[Lemma 3.2]{zhang2016local}, we have the following chain rule: for $f\in W^{1,2}_{loc}\cap L^{\infty}$ and $\phi\in C^2(\mathbb{R})$, it holds that $\phi(f)\in W^{1,2}_{loc}\cap L^{\infty}$ and that
\begin{equation}
    L\left( \phi(f)\right)
    =
    \phi'(f)\cdot Lf + \phi''(f)|\nabla f|^2\cdot m.
\end{equation}

\begin{lemma}\label{lemma:Kato_ineq}
    (Kato's inequality)
    Let $(X,d,m)$ be a proper $\mathrm{RCD}(K,\infty)$ space and $w$ be in $W^{1,2}\cap L^{\infty}$.
    Suppose $f$ in $W^{1,2}_{loc}\cap L^{\infty}$ is such that $L_wf$ is a signed Radon measure and that $L^{s}_{w}f\geq 0$.
    Then $L_w(f_{+})$ is a signed Radon measure such that
    \begin{equation}
        L_w(f_+) \geq \chi_{\{f> 0\}}L_{w}^{ac}f \cdot m_w,
    \end{equation}
    where $f_+:=\max\{f,0\}$ and $f_{-}:=\max\{-f, 0\}$.
\end{lemma}
\begin{proof}
    It suffices to prove that $L_w(|f|)$ is a signed Radon measure and that
    \begin{equation}\label{eq:Kato_ineq}
        L_{w}(|f|)\geq \mathrm{sgn}(f)\cdot L_wf,
    \end{equation}
    where $\mathrm{sgn}(t)=1$ if $t>0$ and $\mathrm{sgn}(t)=-1$ if $t<0$ and $\mathrm{sgn}(t)=0$ if $t=0$.
    Indeed, if the inequality \eqref{eq:Kato_ineq} holds, then it follows that both $L_w(f_{+})$ and $L_w(f_{-})$ are signed Radon measures and \eqref{eq:Kato_ineq} implies that
    \begin{equation}\label{eq:Kato_ineq_2}
        L_w(f_+) + L_w(f_{-})
        \geq
        \chi_{\{f>0\}}\cdot L_w f - \chi_{\{f<0\}}\cdot L_w f.
    \end{equation}
    By locality of the minimal weak upper gradient and the inner regularity of Radon measures, it is immediate to check that $L_w(f_{+})$ is concentrated on the set $\{f\geq 0\}$.
    Then the inequality \eqref{eq:Kato_ineq_2} with the assumption that $L^{s}_{w}f \geq 0$ implies that
    \begin{equation}
        L_{w}(f_{+})
        \geq
        \chi_{\{f>0\}}L_{w}f
        =
        \chi_{\{f>0\}}\left( L^{ac}_wf\cdot m_w + L^{s}_wf\right)
        \geq
        \chi_{\{f>0\}}L^{ac}_w f\cdot m_w.
    \end{equation}
    We are left to show \eqref{eq:Kato_ineq}.
    Let $\varepsilon>0$ and $\phi_{\varepsilon}(t):=\sqrt{t^2+\varepsilon^2}-\varepsilon$ and $f_{\varepsilon}:=\phi_{\varepsilon}(f)$.
    Since $\phi_{\varepsilon}$ is in $C^2(\mathbb{R})$, it follows that $f_{\varepsilon}\leq |f|$ and that
    \begin{equation}\label{eq:Kato_ineq_3}
        |\nabla f_{\varepsilon}|
        =
        |\phi'_{\varepsilon}(f)||\nabla f|
        =
        \frac{|f|}{\sqrt{f^2+\varepsilon^2}}|\nabla f|
        \leq
        |\nabla f|.
    \end{equation}
    Further, the chain rule of $L_w$ for $f\in W^{1,2}_{loc}\cap L^{\infty}$ and $\psi(t)=t^2$ yields
    \begin{multline}
        2f\cdot L_{w}f + 2|\nabla f|^2\cdot m_w
        =
        L_{w}f^2
        =
        L_{w}\left((f_{\varepsilon}+\varepsilon)^2 -\varepsilon^2 \right)\\
        =
        2(f_{\varepsilon}+\varepsilon)L_wf_{\varepsilon} + 2|\nabla f_{\varepsilon}|^2 \cdot m_w.
    \end{multline}
    So by \eqref{eq:Kato_ineq_3}, it follows that
    \begin{equation}\label{eq:Kato_ineq_4}
        L_w f_{\varepsilon}
        \geq
        \frac{f}{f_{\varepsilon}+\varepsilon}\cdot L_w f.
    \end{equation}
    Now, let $\Omega\subseteq X$ be any arbitrary bounded open subset.
    Let $\eta\in \mathrm{Lip}_c$ be a positive cut-off function such that $\eta\equiv 1$ on $\Omega$.
    Since $|\nabla f_{\varepsilon}|\leq |\nabla f|$ and $0\leq f_{\varepsilon}\leq |f|$, it follows that $\eta f_{\varepsilon}\leq \eta |f|$ and that
    \begin{equation*}
        |\nabla (\eta f_{\varepsilon})|
        \leq
        f_{\varepsilon}|\nabla \eta| + \eta |\nabla f_{\varepsilon}|
        \leq
        |f||\nabla \eta| + \eta |\nabla f|.
    \end{equation*}
    Since $f$ is in $W^{1,2}_{loc}\cap L^{\infty}$, it follows that $(\eta f_{\varepsilon})_{\varepsilon}$ is uniformly bounded in $W^{1,2}$.
    Note that by the assumption on $w$, both the Sobolev spaces $W^{1,2}$ and $W^{1,2}(X,m_w)$ and the minimal weak upper gradients induced by $m$ and $m_w$ coincide (see \cite[Lemma 4.11]{ambrosio2014calculus}) and hence $(\eta f_{\varepsilon})_{\varepsilon}$ is also uniformly bounded in $W^{1,2}(X,m_w)$.
    Since $W^{1,2}$ is a Hilbert space, $W^{1,2}(X,m_w)$ is also Hilbert and hence reflexive.
    Therefore, there exists a subsequence $(\eta f_{\varepsilon_j})$ with $\varepsilon_j\searrow 0$ converging weakly in $W^{1,2}(X,m_w)$ to some $g\in W^{1,2}(X,m_w)$.
    As $f_{\varepsilon_j}\rightarrow |f|$ pointwise, we obtain that $g=\eta|f|$ $m_w$-a.e. and in particular $g=|f|$ $m_w$-a.e. on $\Omega$ since $\eta\equiv 1$ on $\Omega$.
    Since $f_{\varepsilon_j}(x)+\varepsilon_j \rightarrow |f|(x)$ for all $x$ in $X$ and $|f|/(f_{\varepsilon_j}+\varepsilon_j)\leq 1$, for any non-negative Lipschitz $\phi\in \mathrm{Lip}_c(\Omega)$, we obtain that
    \begin{multline}\label{eq:Kato_ineq_5}
        L_{\Omega,f}(\phi)
        :=
        -\int_{\Omega} \langle \nabla \phi, \nabla |f| \rangle dm_w
        =
        -\int_{\Omega} \langle \nabla \phi, \nabla (\eta|f|) \rangle dm_w\\
        =
        -\lim_{j\rightarrow \infty} \int_{\Omega} \langle \nabla \phi, \nabla (\eta f_{\varepsilon_j)} \rangle dm_w
        =
        -\lim_{j\rightarrow \infty} \int_{\Omega} \langle \nabla \phi, \nabla f_{\varepsilon_j} \rangle dm_w\\
        =
        \lim_{j\rightarrow \infty} L_wf_{\varepsilon_j}(\phi)
        \geq
        \lim_{j\rightarrow \infty}\int_{\Omega}\frac{\phi f}{f_{\varepsilon_j}+\varepsilon_j}dL_wf
        =
        \int_{\Omega} \phi \mathrm{sgn}(f)dL_wf.
    \end{multline}
    Following a similar argument as in \citep[(4-28),(4-29) in Theorem 4.14]{cavalletti2020new}, there exists a constant $C_{\Omega,f}>0$ such that $|L_{\Omega,f}(\phi)|\leq C_{\Omega,f}\max|\phi|$ for any Lipschitz function $\phi$ with $\mathrm{supp}(\phi)\subset \Omega$.
    Since $X$ is proper, by the Riesz representation theorem, there exists a signed Radon measure $\mu_{\Omega,f}$ on $\Omega$ such that $L_{\Omega,f}(\phi)=\int_{\Omega}\phi d\mu_{\Omega,f}$ for all $\phi$ in $\mathrm{Lip}_{c}(\Omega)$ and that $\mu_{\Omega,f}\geq \mathrm{sgn}(f)L_wf$ on $\Omega$ (see the remark before Theorem 1.2 in \citep{zhang2016local} or \citep[Proposition 6.2.16]{gigli2020lectures} for finite signed Radon measures).
    Clearly $\mu_{\Omega_1,f}$ and $\mu_{\Omega_2,f}$ coincide on $\Omega_1\cap \Omega_2$ for any bounded open subsets $\Omega_1$ and $\Omega_2$.
    Henceforth, there exists a unique signed Radon measure $\nu$ on $X$ such that $\nu|_{\Omega}=\mu_{\Omega,f}$ for all bounded open domain $\Omega$, and thus, we obtain that $L_w|f|$ is a signed Radon measure with $L_w|f| \geq \mathrm{sgn}(f)\cdot L_wf$.
\end{proof}

We now show the Omori-Yau type maximum principle in proper $\mathrm{RCD}(K,\infty)$ spaces based on the previous Kato's inequality.
While the proof in \textcolor{blue}{\citep{zhang2016local}} relies on the doubling property and the weak Poincaré inequality of the underlying metric measure space as well as the weak maximum principle, our proof is based on the "Sobolev-to-Lip" property of $\mathrm{RCD}(K,\infty)$ spaces which in general does not support the doubling property.
\begin{lemma}\label{lemma:maximum_principle}
    (Omori\--Yau type maximum principle)
    Let $(X,d,m)$ be a proper $\mathrm{RCD}(K,\infty)$ space with $K$ in $\mathbb{R}$.
    Let further $f$ be in $W^{1,2}\cap L^{\infty} \cap D(\bm{\Delta})$ such that $\bm{\Delta}^{s}f\geq 0$.
    Suppose that $f$ achieves one of its strict maximum in $X$ in the sense that there exists a bounded and measurable subset $U\subset X$ satisfying $m(U)>0$ and $m(X\setminus U)>0$ with
    \begin{equation}
        \esssup_{U}f > \esssup_{X\setminus U}f.
    \end{equation}
    Then, given any $w$ in $W^{1,2}\cap \mathrm{Lip}_b$, for any $\varepsilon>0$, we have
    \begin{equation}\label{eq:MP_1}
        m\left( \left\{x\in X \colon f(x)\geq \esssup_{X}f-\varepsilon \text{ and } (\bm{\Delta}^{ac}f)(x) + \langle \nabla f, \nabla w \rangle(x)\leq \varepsilon   \right\}\right)>0.
    \end{equation}
    In particular, there exists a sequence $(x_j)$ in $X$ such that $f(x_j)\geq \esssup f(x_j) -1/j$ and $(\bm{\Delta}^{ac}f)(x_j)+\langle \nabla f,\nabla w\rangle(x_j)\leq 1/j$.
\end{lemma}
\begin{proof}
    We adapt the proof in \citep{zhang2016local}.
    Let $M:=\esssup_{X}f$.
    Suppose by contradiction that there exists $\varepsilon_0>0$ and $w$ in $W^{1,2}\cap \mathrm{Lip}_b$ such that \eqref{eq:MP_1} fails.
    Then for possibly smaller $\varepsilon_0>0$ such that $M-\varepsilon_0>\esssup_{X\setminus U}f$, it follows that $g:=(f-(M-\varepsilon_0))_{+}$ is in $W^{1,2}$ with $g= 0$ $m$-a.e. in $X\setminus U$, and that
    \begin{equation}
        m\left(\left\{x\in X: f(x)>M-\varepsilon_0\text{ and } \bm{\Delta}^{ac}f(x)+\langle \nabla f, \nabla w \rangle(x)\leq \varepsilon_0  \right\}\right)
        =0.
    \end{equation}
    It follows that for $m$-a.e. $x$ in $\{y\in X: f(y)>M-\varepsilon_0\}$, we have
    \begin{equation}
        \bm{\Delta}^{ac}f(x) + \langle \nabla f, \nabla w \rangle(x)>\varepsilon_0.
    \end{equation}
    Note that since $w\in W^{1,2}\cap \mathrm{Lip}_b$ and $\bm{\Delta}f$ is a signed Radon measure, it follows that $e^w(\bm{\Delta}f + \langle \nabla w, \nabla f\rangle m)$ is a well-defined signed Radon measure.
    By the identity of the weighted weak Laplacian $L_wf=e^w(Lf + \langle \nabla w, \nabla f \rangle m)$ and $Lf=\bm{\Delta}f$, it follows that
    \begin{equation}
        L^{ac}_wf \cdot m_w
        =
        e^{w}\left( L^{ac}f + \langle \nabla w, \nabla f\rangle \right)\cdot m
        \geq
        e^{-\|w\|_{L^\infty}}\varepsilon_0\cdot m
        >0
    \end{equation}
    on $\{y\in X: f(y)> M-\varepsilon_0\}$.
    Moreover, the identity of the weighted weak Laplacian together with $\bm{\Delta}^{s}f\geq 0$ implies that $L^{s}_{w}f \geq 0$ and $f-(M-\varepsilon_0)$ is in $W^{1,2}_{loc}\cap L^{\infty}$, applying Lemma \ref{lemma:Kato_ineq} to $f-(M-\varepsilon_0)$, yields
    \begin{equation}\label{eq:maximum_Principle_1}
        L_w g
        =
        L_w\left(f-(M-\varepsilon_0)\right)_+
        \geq
        \chi_{\{f> M-\varepsilon_0\}}L^{ac}_{w}f\cdot m_w
        \geq
        0.
    \end{equation}
    By the definition of $L_w$, it follows from \eqref{eq:maximum_Principle_1} that
    \begin{equation}
        -\int_X \langle \nabla g, \nabla g \rangle dm_w
        =
        L_wg(g)
        \geq 0,
    \end{equation}
    which implies that $|\nabla g|=0$ $m_w$-a.e.
    From $m$ being equivalent to $m_w$ since $w\in L^{\infty}$, we get that $|\nabla g|=0$ $m$-a.e.
    Together with $g$ being in $W^{1,2}$ and $|\nabla g|$ in $L^{\infty}$, by the Sobolev-to-Lip property of $\mathrm{RCD}(K,\infty)$ spaces, it follows that $g$ admits a Lipschitz representation $\tilde{g}$ with $\mathrm{Lip}(\tilde{g})\leq \| |\nabla g|\|_{\infty}$.
    Then, from the fact that $|\nabla g|=0$ $m$-a.e. as well as $g=0$ $m$-a.e on $X\setminus U$, it follows that $\tilde{g}$ is constant and $\tilde{g}\equiv 0$.
    Hence, $f\leq M-\varepsilon_0$ $m$-a.e., which is a contradiction.
\end{proof}

\section{Proof of Rellich-Kondrachov type theorem}\label{appendix_B}
In this subsection, we show the second claim of Lemma \ref{lemma:Pre_Sobolev_embedding}.
The Rellich\--Kondrachov type theorem has been previously proved for $\mathrm{RCD}^*(K,N)$ spaces with $K>0$ and $N\in (2,\infty)$ in \cite{profeta2015sharp} by following the argument in \cite[Theorem 8.1]{hajlasz2000sobolev}.
A careful inspection of their proof shows that the result also holds for compact $\mathrm{CD}(K,N)$ spaces with $K\in \mathbb{R}$ and $N\in (2,\infty)$ without the infinitesimal Hilbertianty assumption.
For the sake of completeness, we provide the proof.
\begin{proof}[Proof of \ref{lemma:Rellich_Kondrachov} of Lemma \ref{lemma:Pre_Sobolev_embedding}]
    We follow the similar arguments from \cite{hajlasz2000sobolev}.
    Let $D:=\mathrm{diam}(X)$.
    By the assumptions and the generalized Bishop\--Gromov inequality, without loss of generality, we can assume that $m(X)=1$.
    By the Sobolev inequality \eqref{eq:Sobolev_ineq} and the assumption that $\sup_{n}\|f_n\|_{W^{1,2}}<\infty$, it follows that $(f_n)_n$ is bounded in $L^{2^*}$.
    Since $L^{2^*}$ is reflexive, there exists $f\in L^{2^*}$ such that there exists a subsequence $(f_{n_k})_k$ of $(f_n)_n$ such that $f_{n_k}$ converges to $f$ weakly in $L^{2^*}$.
    It suffices to show that $f_{n_k}$ converges to $f$ strongly in $L^q$ for all $q\in [1,2^*)$.
    For the notation's convenience, we relabel $(f_{n_k})_k$ as $(f_k)_k$.

    Let $\varepsilon>0$ and $0<r\leq D/2$ be arbitrary and $q\in [1,2^*)$.
    Let $\bar{f}_{r}(x):=\fint_{B(x,r)}f dm$ and $\bar{(f_k)}_{r}(x):=\fint_{B(x,r)}f_k dm$ for any $x\in X$.
    Then it follows that
    \begin{multline}\label{eq:RK_embed_1}
        m\left(\left|f_k(x)-f(x) \right|>\varepsilon\right)
        \leq
        m\left( \left| f_k(x) - \bar{(f_k)}_r(x)\right|>\varepsilon/3\right)\\
        +
        m\left( \left| \bar{(f_k)}_r(x) - \bar{f}_r(x)\right|>\varepsilon/3\right)
        +
        m\left( \left| \bar{f}_r(x) - f(x)\right|>\varepsilon/3\right).
    \end{multline}
    Note that the compact $\mathrm{CD}(K,N)$ space with bounded diameter $D$ support the global doubling property and the weak $(1,1)$-Poincaré inequality \eqref{eq:Poincare_ineq} with the global constant $C(K,N,D)$.
    By the equivalent characterizations of Poincaré inequalities \cite[Theorem 2]{keith2003modulus} and \cite[Proposition 4.13]{bjorn2011nonlinear}, it follows that the weak $(1,2)$-Poincaré inequality \eqref{eq:Poincare_ineq} holds for all pairs $(f_k, |\nabla f_k|)_k$.
    Moreover, by the doubling property of $\mathrm{CD}(K,N)$ spaces, compact $\mathrm{CD}(K,N)$ spaces support the Lebesgue differentiation theorem and the maximal theorem (see \cite{heinonen2015sobolev}). 
    Let $x\in X$ be any Lebesgue point of $X$ and $r_i:=r/2^i$ for $i\in \mathbb{N}$.
    Then for the first term of the right-hand side of \eqref{eq:RK_embed_1}, it follows that
    \begin{multline}
        \left| f_k(x) -\bar{(f_k)}_r(x)\right|
        \leq
        \sum_{i=0}^{\infty}\left| \bar{(f_k)}_{r_i}(x) - \bar{(f_k)}_{r_{i+1}}(x)\right|\\
        \leq
        \sum_{i=0}^{\infty}\fint_{B(x,r_i)}\left|f_k(y) - \bar{(f_k)}_{r_i}(x)\right| dm(y)
        \leq
        C(D)\sum_{i=0}^{\infty}r_i\left(\fint_{B(x, 2r_i)}|\nabla f_k|^2 dm \right)^{1/2} \\
        \leq
        2C(D)r \cdot M\left(|\nabla f_k|^2\right)^{1/2}(x),
    \end{multline}
    where $Mf(x):=\sup_{r>0}\fint_{B(x,r)}|f(y)| dm$ be the centered maximal operator.
    So by applying the maximal theorem \cite[Theorem 3.5.6]{heinonen2015sobolev}, it follows that
    \begin{multline}\label{eq:RK_embed_2}
        m\left(\left| f_k(x) - \bar{(f_k)}_r(x)\right|>\varepsilon/3\right)
        \leq
        m\left(M\left(\left|\nabla f_k\right|^2\right)^{1/2}(x)> \frac{\varepsilon}{6C(D)r}\right)\notag\\
        \leq
        \frac{36 C_1\cdot C(D)^2r^2}{\varepsilon^2}\sup_{k}\left\| \left|\nabla f_{k} \right| \right\|^2_{2}.
    \end{multline}
    For the second term of the right-hand side of \eqref{eq:RK_embed_1}, note that since $f_k$ converges to $f$ weakly in $L^{2^*}(X)$, so taking the test function $\eta:=\chi_{B(x,r)}/m(B(x,r))$, it follows that
    \begin{equation}
        \frac{1}{m(B(x,r))}\int_{B(x,r)}f_k dm \rightarrow \frac{1}{m(B(x,r))}\int_{B(x,r)}f dm,\quad \text{as $k\rightarrow \infty$},
    \end{equation}
    which implies that $\bar{(f_k)}_r$ converges to $\bar{f}_r$ pointwise and thus $\bar{(f_k)}_r$ converges to $\bar{f}$ in probability.
    For the third term of right-hand side of \eqref{eq:RK_embed_1}, applying the Lebesgue differentiation theorem, it follows that $\bar{f}_r$ converges to $f$ $m$-a.e. on $X$ as $r\searrow 0$.
    Together with all arguments above, taking $k\rightarrow \infty$ first and then taking $r\searrow 0$ on both side of \eqref{eq:RK_embed_1}, it follows that
    \begin{equation}
        \lim_{k\rightarrow \infty}m\left(\left| f_k(x)-f(x)\right|>\varepsilon\right)
        =0.
    \end{equation}
    Then by \cite[Lemma 8.2]{hajlasz2000sobolev}, we obtain that $f_k\rightarrow f$ strongly in $L^q$ for all $q\in [1,2^*)$.

\end{proof}

\end{appendix}

\bibliographystyle{plainnat}
\bibliography{biblio}
\end{document}